\documentclass[12pt]{amsart}
\usepackage{bm}
 \usepackage{graphicx}

 \newtheorem{thm}{Theorem}[section]
 
 \newtheorem{lem}[thm]{Lemma}
 \newtheorem{prop}[thm]{Proposition}
 \theoremstyle{definition}
 \newtheorem{defn}[thm]{Definition}
 \theoremstyle{remark}
 \newtheorem{rem}[thm]{Remark}
 \newtheorem{assume}[thm]{Assumption}
 \numberwithin{equation}{section}

 \newcommand{\Real}{\mathbb{R}}

\begin{document}

\title[Curvature and hyperbolicity]{On curvature and hyperbolicity of monotone Hamiltonian systems}

\author{Paul W.Y. Lee}
\email{wylee@math.cuhk.edu.hk}
\address{Room 216, Lady Shaw Building, The Chinese University of Hong Kong, Shatin, Hong Kong}

\date{\today}

\begin{abstract}
Assume that a Hamiltonian system is monotone. In this paper, we give several characterizations on when such a system is Anosov. Assuming that a monotone Hamiltonian system has no conjugate point, we show that there are two distributions which are invariant under the Hamiltonian flow. We show that a monotone Hamiltonian flow without conjugate point is Anosov if and only if these distributions are transversal. We also show that if the reduced curvature of the Hamiltonian system is non-positive, then the flow is Anosov if and only if the reduced curvature is negative somewhere along each trajectory. This generalizes the corresponding results on geodesic flows in \cite{Eb}.
\end{abstract}

\maketitle

\section{Introduction}

In this paper, we consider when a Hamiltonian system is Anosov. Let us first recall the definition of a Anosov flow. Let $X$ be a vector field defined on a manifold $N$. Its flow is Anosov if there is a Riemannian metric $\left<\cdot,\cdot\right>$ and a splitting $TN=\Real X\oplus\Delta^+\oplus\Delta^-$ of the tangent bundle $TN$ of $N$ such that the followings hold.
\begin{enumerate}
\item $\Delta^\pm$ are distributions which are invariant under the flow $\varphi_t$ of $X$,
\item there are positive constants $c_1$ and $c_2$ such that $|d\varphi_{\pm t}(v)|\leq c_1e^{-c_2t}|v|$ for all $v$ in $\Delta^\pm$ and for all $t\geq 0$.
\end{enumerate}

In \cite{An}, it was shown that the geodesic flow on the unit sphere bundle of a compact manifold is Anosov if the sectional curvature of the manifold is everywhere negative. This result was generalized to monotone Hamiltonian systems in \cite{Ag2} using the curvature invariants introduced in \cite{Ag1}. On the other hand, it was shown in \cite{Eb} that there are many alternative characterizations of Anosov geodesic flow under the assumption that the flow has no conjugate point. Some of them were extended by \cite{CoIt} to Hamiltonian systems arising from the classical action functionals in calculus of variations.

In this paper, we extend the results in \cite{Eb,CoIt} to monotone Hamiltonian systems. Let us first recall the definition and the setup. Let $M$ be a manifold equipped with a symplectic structure $\omega$ and a Lagrangian distribution $\Lambda$. Let $H:M\to\Real$ be a fixed Hamiltonian and let us denote the corresponding Hamiltonian vector field by $\vec H$. Recall that $\vec H$ is defined by $\omega(\vec H,\cdot)=-dH(\cdot)$. Let $V_1$ and $V_2$ be two sections of $\Lambda$ and let $\left<V_1,V_2\right>$ be defined by
\[
\left<V_1,V_2\right>=\omega([\vec H,V_1],V_2).
\]
on $\Lambda$. It is not hard to see that $\left<\cdot,\cdot\right>$ defines a symmetric bilinear form on the distribution $\Lambda$. The Hamiltonian vector field $\vec H$ is monotone if $\left<\cdot,\cdot\right>$ defines a Riemannian metric on $\Lambda$.

The monotonicity of a Hamiltonian vector field $\vec H$ means essentially that the restriction of $H$ to each space $\Lambda_\alpha$ is strictly convex. More precisely, let $H:T^*N\to\Real$ be a Hamiltonian defined on the cotangent bundle $M=T^*N$ of a manifold $N$. Assume that the Hamiltonian is fibrewise strictly convex. That is $H|_{T^*N_x}$ is strictly convex for each $x$ in $N$. Then the Hamiltonian vector field $\vec H$ is monotone if $T^*N$ is equipped with the symplectic structure $\omega=d\theta$, where $\theta$ is the tautological one form defined by $\theta_\alpha(V)=\alpha(d\pi(V))$. These are the Hamiltonians considered in \cite{CoIt}. For example, if $H$ is given by the kinetic energy of a Riemannian metric on $N$, then $H$ is fibrewise strictly convex and the Hamiltonian vector field $\vec H$ is monotone. In this case, the flow of $\vec H$ is the geodesic flow.

More generally, one can consider twisted symplectic structure defined on $T^*N$ by $\omega=d\theta+\pi^*\eta$, where $\eta$ is any closed two form on $N$. Then the Hamiltonian vector field is monotone with respect to this twisted symplectic structure if and only if the Hamiltonian is fibrewise strictly convex. Note also that if $\vec H$ is monotone with respect to a symplectic structure $\omega$ and a Lagrangian distribution $\Lambda$, then we can slightly perturb the structures $(\omega,\Lambda)$ and $\vec H$ is still monotone.

Let $\alpha$ be a point in $M$. We say that the Hamiltonian flow $\varphi_t$ of $\vec H$ has no point conjugate to $\alpha$ if $d\varphi_t(\Lambda_\alpha)$ intersects transversely with $\Lambda_{\varphi_t(\alpha)}$ for all time $t$. If the Hamiltonian flow has no conjugate point, then there are two (measurable) distributions which are invariant under the Hamiltonian flow. This was first proved, in the case of the geodesic flow, in \cite{Gr} (see also \cite{CoIt} for an extension).

\begin{thm}\label{intro1}
Assume that the Hamiltonian vector field $\vec H$ is monotone and its flow $\varphi_t$ does not contain any conjugate point on $M$. Then there are (measurable) Lagrangian distributions $\Delta^\pm$ of $M$ which are invariant under $d\varphi_t$.
\end{thm}

By the work of \cite{AgGa}, we can define the reduced curvature $\mathfrak{\tilde R}$ of a monotone Hamiltonian vector field $\vec H$ (see Section \ref{Reduction} for the definition). Under the assumptions of Theorem \ref{intro1}, we can show that the integral of the trace $\tilde{\mathfrak r}$ of $\tilde{\mathfrak R}$ with respect to any invariant measure (in particular the Liouville measure) of $\vec H$ is non-positive. Moreover, this integral vanishes only if $\tilde{\mathfrak r}$ vanishes. This extends the results of \cite{Ho,Gr,In} to our setting. More precisely, we have

\begin{thm}\label{totalreduce}
Let $c$ be a regular value of $H$ and assume that the Hamiltonian vector field $\vec H$ is monotone. Let $\mu$ be an invariant measure of the flow $\varphi_t$ of $\vec H$ on $\Sigma_c:=H^{-1}(c)$. Assume that $\Sigma_c$ is compact and $\varphi_t$ has no conjugate point on $\Sigma_c$. Then the following holds
\[
\int_{\Sigma_c}\mathfrak r_\alpha \,d\mu(\alpha)\leq 0.
\]
Moreover, equality holds only if $\mathfrak r\equiv 0$ on the support of $\mu$.
\end{thm}

We also show that the flow of the Hamiltonian vector field $\vec H$ is Anosov assuming that the reduced curvature is negative.

\begin{thm}\label{anosov}
Let $c$ be a regular value of $H$. Assume that the Hamiltonian vector field is monotone and the reduced curvature is bounded above and below by two negative constants on $\Sigma_c$. Then the flow of $\vec H$ is Anosov on $\Sigma_c$.
\end{thm}

The above theorem is proved in \cite{Ag2} under the assumption that $\Sigma_c$ is compact. We give a different proof which relaxes this compactness assumption to a lower curvature bound.

If the invariant distributions $\Delta^\pm$ defined in Theorem \ref{intro1} are everywhere transversal, then it was shown in \cite{Eb} that the geodesic flow is Anosov. An extension of this result can also be found in \cite{CoIt}. By combining a reduction procedure together with the analysis in \cite{Eb}, we obtain the following result.

\begin{thm}\label{main2}
Suppose that the assumption of Theorem \ref{intro1} are satisfied. Let $c$ be a regular value of $H$ and assume that $\Sigma_c=H^{-1}(c)$ is compact. Then the following are equivalent.
\begin{enumerate}
\item The flow $\varphi_t$ is Anosov on $\Sigma_c$,
\item $\Delta^+$ and $\Delta^-$ are transversal in $T\Sigma_c$,
\item $\Delta^+\cap\Delta^-=\textbf{span}\{\vec H\}$.
\end{enumerate}
\end{thm}

Under the assumption that the reduced curvature is everywhere non-positive, we also obtain the following which generalize another result of \cite{Eb}.

\begin{thm}\label{nonpositive}
Let $c$ be a regular value of $H$ and assume that $\Sigma_c=H^{-1}(c)$ is compact.
Assume that the monotone Hamiltonian vector field $\vec H$ has non-positive reduced curvature on $\Sigma_c$. Then the flow $\varphi_t$ of $\vec H$ is Anosov on $\Sigma_c$ if and only if, for each $\alpha$ in $\Sigma_c$, there is a time $t$ such that the reduced curvature $\tilde {\mathfrak R}$ of $\vec H$ satisfies $\left<\tilde{\mathfrak R}_{\varphi_t(\alpha)}\tilde v,\tilde v\right><0$ for some $t$ and for some vector $\tilde v$ in $\tilde\Lambda_{\varphi_t(\alpha)}$.
\end{thm}

Using the result in Theorem \ref{main2}, we can estimate the measure theoretic entropy for invariant measures of $\varphi_t$ in terms of the reduced curvature $\tilde{\mathfrak R}$. This generalizes the corresponding results in \cite{FrMa} and \cite{In}.

\begin{thm}\label{entropy1}
Let $c$ be a regular value of $H$. Let $\mu$ be an invariant measure of the flow $\varphi_t$ of $\vec H$ on the compact manifold $\Sigma_c:=H^{-1}(c)$. Assume that $\varphi_t$ has no conjugate point on the support of $\mu$. Then the following holds
\[
h_\mu(\varphi_t) \leq (n-1)^{1/2}\left(-\int_{\Sigma_c}\textbf{tr}\,\tilde{\mathfrak R}_\alpha d\mu(\alpha)\right)^{1/2}.
\]
Moreover, equality holds only if $\tilde{\mathfrak R}$ is constant on the support of $\mu$.
\end{thm}

We remark that a lower estimate under the assumption that the reduced curvature is non-positive was done in \cite{Ch} which generalizes the earlier work of \cite{BaWo,OsSa}.

Finally, we also show that the following generalization of the result in \cite{PaPe} is also possible.

\begin{thm}\label{entropy2}
Let $c$ be a regular value of $H$. Let $\mu$ be an invariant measure of the flow $\varphi_t$ of $\vec H$ on the compact manifold $\Sigma_c:=H^{-1}(c)$. Then the following holds
\[
h_\mu(\varphi_t) \leq \frac{1}{2}\int_{\Sigma_c}\sum_{i=1}^{n-1}|1-\lambda_i(\alpha)|d\mu(\alpha)
\]
for any invariant measure $\mu$ of $\varphi_t$ on $\Sigma_c$ and where $\lambda_i(\alpha)$ are eigenvalues of the operator $\tilde{\mathfrak R}_\alpha$.
\end{thm}

The content of this paper is as follows. In Section \ref{Regular}, we discuss some materials on curves in Lagrangian Grassmannian which are needed in the definition of the curvature of $\vec H$. In section \ref{On}, we recall several basic results on linear second order ODEs which are needed in this paper. In Section \ref{Monotone}, we recall the definition of the curvature of $\vec H$. In Section \ref{Reduction}, we recall a reduction procedure studied in \cite{AgChZe} which is needed for the proof of the above theorems. Finally, sections \ref{Existence}-13 are devoted to the proofs.

\smallskip

\section{Notations}

$\mathcal V$ a symplectic vector space

$M$ a symplectic manifold

$\omega$ symplectic form on $\mathcal V$ or on $M$

$\mathcal L(\mathcal V)$ Lagrangian Grassmannian of $\mathcal V$

$J$ curve in $\mathcal L(\mathcal V)$

$J^o$ derivative curve of $J$

$R$ curvature operator of $J$

$\mathcal R$ matrix representation of $R$

$\left<\cdot,\cdot\right>^t$ the canonical bilinear form on $J(t)$

$e^1(t),...,e^n(t)$ a canonical frame of a regular curve $J$

$f^i(t)=\dot e^i(t)$

$H$ Hamiltonian

$\vec H$ Hamiltonian vector field of $H$

$J_\alpha$ Jacobi curve of $\vec H$ at $\alpha$

$R_\alpha(t)$ curvature operator of $J_\alpha$

$\mathfrak R$ curvature operator of $\vec H$

$\tilde J_\alpha$ reduced Jacobi curve of $\vec H$ at $\alpha$

$\tilde R_\alpha(t)$ curvature operator of $\tilde J_\alpha$

$\tilde{\mathfrak R}$ reduced curvature operator of $\vec H$

\smallskip

\section{Regular Curves in Lagrangian Grassmannian}\label{Regular}

Let $\mathcal V$ by a $2n$-dimensional vector space equipped with a symplectic form $\omega$. Recall that a $n$-dimensional subspace $\Delta$ of the symplectic vector space $\mathcal V$ is Lagrangian if the restriction of $\omega$ to $\Delta$ vanishes. The space of all Lagrangian subspaces in $\mathcal V$, called Lagrangian Grassmannian, is denoted by $\mathcal L=\mathcal L(\mathcal V)$. In this section, we recall the definition and properties of regular curves in $\mathcal L$. For a more complete discussion, see \cite{Ag2,LiZe}.

A smooth curve $t\mapsto J(t)$ in $\mathcal L$ carries a family of canonical bilinear forms $\left<\cdot,\cdot\right>^t$ defined by
\begin{equation}\label{bilinear1}
\left<v_1,v_2\right>^t:=\omega(\dot v_1(t),v_2)
\end{equation}
for all $v_1$ and $v_2$ in $J(t)$, where $\tau\mapsto v(\tau)$ is a curve satisfying $v_1=v(t)$ and $v(\tau)\in J(\tau)$ for each $\tau$.

\begin{defn}
A smooth curve $t\mapsto J(t)$ in the Lagrangian Grassmannian $\mathcal L$ is \textit{regular} if the bilinear form (\ref{bilinear1}) is non-degenerate for each $t$.
\end{defn}

Recall that a basis $e_1,...,e_n,f_1,...,f_n$ of the symplectic vector space $\mathcal V$ is a Darboux basis if $\omega(e_i,e_j)=\omega(f_i,f_j)=0$ and $\omega(f_i,e_j)=\delta_{ij}$. For a regular curve $J$, one can also define a canonical frame in $J$ which is unique up to transformations by orthogonal matrices. In fact, canonical frames can be found for more general curves, see \cite{LiZe} for more detail.

\begin{prop}\label{Darboux}
Assume that $J$ is a regular curve in the Lagrangian Grassmannian $\mathcal L$. Then there exists a smooth family of bases
\[
E(t)=(e^1(t),...,e^n(t))^T \quad (T \text{ denotes transpose})
\]
on $J(t)$ orthonormal with respect to the inner product $\left<\cdot,\cdot\right>^t$ such that, for each time $t$,
\[
\{e^1(t),...,e^n(t),\dot e^1(t),...,\dot e^n(t)\}
\]
forms a Darboux basis of the symplectic vector space $\mathcal V$. Moreover, if $\bar E(t)=(\bar e^1(t),...,\bar e^n(t))^T$ is another such family, then there exists an orthogonal matrix $U$ (independent of time $t$) such that $\bar E(t)=U\,E(t)$.
\end{prop}

\begin{rem}
For the rest of the paper, we will set $f^i(t)=\dot e^i(t)$.
\end{rem}

\begin{proof}
Let us fix a family of bases $\tilde E(t)=(\tilde e^1(t),...,\tilde e^n(t))^T$ on $J (t)$ orthonormal with respect to the canonical inner product. Since $J(t)$ is a Lagrangian subspace, we have
\begin{equation}\label{almostDarboux1}
\omega(\tilde e^i(t),\tilde e^j(t))=0.
\end{equation}

Since $\tilde E(t)$ is orthonormal with respect to the inner product $\left<\cdot,\cdot\right>^t$, we also have
\begin{equation}\label{almostDarboux2}
\omega(\dot{\tilde e}^i(t),\tilde e^j(t))=\delta_{ij}.
\end{equation}

Let $U(t)$ be any smooth family of invertible matrices, let $E(t)=U(t)\tilde E(t)$, and let $f^i(t)=\dot e^i(t)$. Let $\Omega(t)$ be the matrix with $ij$-th entry equal to $\omega(\dot e^i(t),\dot e^j(t))$. Then, by (\ref{almostDarboux1}) and (\ref{almostDarboux2}), we have
\[
\Omega(t)=-\dot U(t)U(t)^T+U(t)\dot U(t)^T+U(t)\tilde\Omega(t)U(t)^T.
\]

Therefore, if we let $\tilde \Omega(t)$ be the matrix with $ij$-th entry $\omega(\dot {\tilde e}^i(t),\dot{\tilde e}^j(t))$ and let $U(t)$ be the solution of
\begin{equation}\label{U(t)}
\dot U(t)=\frac{1}{2}U(t)\tilde\Omega(t)
\end{equation}
with $U(0)$ orthogonal, then $\Omega(t)\equiv 0$.

Note that $\tilde\Omega(t)$ is skew-symmetric. Therefore, $\frac{d}{dt}\left(U(t)U(t)^T\right)=0$ and $U(t)$ is orthogonal since $U(0)$ is. Hence by (\ref{almostDarboux2}) and the definition of $E(t)$
\[
e^1(t),...,e^n(t),f^1(t),...,f^n(t)
\]
is a Darboux basis.

Finally, if we assume that $\tilde e^1(t),...\tilde e^n(t),\dot{\tilde e}^1(t),...,\dot{\tilde e}^n(t)$ is a Darboux basis of $\mathcal V$ for each $t$, then $\tilde\Omega(t)\equiv 0$ and hence $\dot U(t)=0$ by (\ref{U(t)}). Therefore, the uniqueness claim follows.
\end{proof}

Proposition \ref{Darboux} leads to the following definition.

\begin{defn}
A family of Darboux bases $e^1(t),...,e^n(t),f^1(t),...,f^n(t)$ of a symplectic vector space $\mathcal V$ is a \textit{canonical frame} of the regular curve $J$ if, for each $i=1,...,n$,
\begin{enumerate}
\item $e^i(t)$ is contained in $J(t)$,
\item $f^i(t)=\dot e^i(t)$,
\end{enumerate}
\end{defn}

The proof of Proposition \ref{Darboux} gives the following result which will be needed later.

\begin{lem}\label{noncanonical}
Let $\tilde E(t):=(\tilde e^1(t),...,\tilde e^n(t))^T$ be a family of orthonormal basis (with respect to $\left<\cdot,\cdot\right>^t$) in a curve $J(\cdot)$ of the Lagrangian Grassmannian $\mathcal L$. Let $\Omega$ be the matrix with $ij$-th entry equal to
$\omega(\dot {\tilde e}^i(t),\dot{\tilde e}^j(t))$. Let $U$ be a solution of
\begin{equation}\label{UOmega}
\dot U(t)=\frac{1}{2}U(t)\Omega(t).
\end{equation}
Then $E(t)=U(t)\tilde E(t)$ and its derivative $\dot E(t)$ forms a canonical frame.
\end{lem}

Proposition \ref{Darboux} also allows us to make the following definitions.

\begin{defn}\label{der}
Let $e^1(t),...,e^n(t),f^1(t),...,f^n(t)$ be a canonical frame of a regular curve $J$. The curve $J^o(\cdot)$ in $\mathcal L(\mathcal V)$ defined by
\[
J^o(t):=\text{span}\{f^1(t),...,f^n(t)\}
\]
is called the \emph{derivative curve} of $J$.
\end{defn}

The canonical frame satisfies a second order equation.

\begin{prop}\label{structural}
Let $e^1(t),...,e^n(t),f^1(t),...,f^n(t)$ be a canonical frame. Then there is a linear operator $R(t):J(t)\to J(t)$ symmetric with respect to the symmetric bilinear form $\left<\cdot,\cdot\right>^t$ such that
\[
\dot e^i(t)=f^i(t),\quad \dot f^i(t)=-R(t)\,e^i(t).
\]
\end{prop}

\begin{proof}
By differentiating the condition $\omega(\dot e^i(t), e^j(t))=\delta_{ij}$ and using the condition $\omega(\dot e^i(t),\dot e^j(t))=0$, we obtain
$\omega(\ddot e^i(t),e^j(t))=0$. Since $J(t)$ is a Lagrangian subspace, $\ddot e^i(t)$ is contained in $J(t)$. Therefore, we can define $R(t)$ by
\[
R(t)\,e^i(t)=-\ddot e^i(t).
\]
By Theorem \ref{Darboux}, this definition of $R(t)$ is independent of the choice of canonical Darboux frames.

Finally, using the equation $\omega(f^i(t),f^j(t))=0$ and differentiating with respect to time, we see that
\[
\omega(\dot e^j(t), R(t)e^i(t))=\omega(\dot e^i(t), R(t)e^j(t)).
\]
It follows that $R(t)$ is symmetric with respect to the canonical inner product.
\end{proof}

\begin{defn}\label{curvatureoperator}
The equations
\[
\dot e^i(t)=f^i(t),\quad \dot f^i(t)+R(t)e^i(t)=0
\]
in Proposition \ref{structural} are called \textit{structural equations} of the curve $J$. The operators
$R(t)$ are the \textit{curvature operators} of $J$. The matrix representation of $R(t)$ is denoted by $\mathcal R(t)$ and it is defined by
\[
R(t)e^i(t)=\sum_{j=1}^n\mathcal R_{ij}(t)e^j(t)
\]
\end{defn}

\smallskip

\section{Monotone Hamiltonian vector fields}\label{Monotone}

Let $M$ be a symplectic manifold equipped with a symplectic structure $\omega$ and a Lagrangian distribution $\Lambda$. Let $H:M\to\Real$ be Hamiltonian and let $\vec H$ be the corresponding Hamiltonian vector field defined by
\[
\omega(\vec H,\cdot)=-dH(\cdot).
\]

Let us consider the canonical symmetric, bilinear form $\left<\cdot,\cdot\right>$ of the Hamiltonian vector field $\vec H$ defined on $\Lambda$ by
\begin{equation}\label{bilinear}
\left<v_1,v_2\right>_\alpha=\omega_\alpha([\vec H,V_1],V_2),
\end{equation}
where $V_1$ and $V_2$ are two sections of $\Lambda$ such that $V_1(\alpha)=v_1$ and $V_2(\alpha)=v_2$. Since $\Lambda$ is a Lagrangian distribution and the Hamiltonian vector field $\vec H$ preserves $\omega$, the above bilinear form is well-defined.

\begin{defn}
We say that the Hamiltonian vector field $\vec H$ is \textit{monotone} if the above bilinear form is a Riemannian metric on $\Lambda$.
\end{defn}

For example, if $M$ is the cotangent bundle $T^*M$ of a manifold $M$ equipped with the standard symplectic form $\omega=\sum_idp_i\wedge dx_i$, where $(x_1,...,x_n,p_1,...,p_n)$ is the local coordinates of $T^*M$. Then the canonical bilinear form $\left<\cdot,\cdot\right>$ corresponding to the Hamiltonian $H$ is given by
\[
\left<\partial_{p_i},\partial_{p_j}\right>=H_{p_ip_j}.
\]
Therefore, in this case, $\vec H$ is monotone if and only if $H$ is fibrewise strictly convex.

In this section, following the approach introduced by \cite{AgGa}, we consider the curvature of monotone Hamiltonian vector fields. For this, let $\varphi_t$ be the flow of the Hamiltonian vector field $\vec H$, let us fix a point $\alpha$ in the manifold $M$ and consider the following curve of Lagrangian subspaces in the Lagrangian Grassmannian $\mathcal L(T_\alpha M)$.

\begin{defn}
The curve $t\mapsto J_\alpha(t)$ in the Lagrangian Grassmannian $\mathcal L(T_\alpha M)$ defined by
\[
J_\alpha(t):=d\varphi_t^{-1}(\Lambda_{\varphi_t(\alpha)})
\]
is called the \textit{Jacobi curve} of $\vec H$ at $\alpha$.
\end{defn}

\begin{prop}\label{regular}
The canonical symmetric, bilinear form (\ref{bilinear}) of the Hamiltonian vector field $\vec H$ and the canonical bilinear form of the Jacobi curve defined by (\ref{bilinear1}) are related by
\[
\left<v_1,v_2\right>_{\varphi_t(\alpha)}=\left<d\varphi_t^{-1}v_1,d\varphi_t^{-1}v_2\right>^t
\]
for all $v_1$ and $v_2$ in $T_{\varphi_t(\alpha)}M$.

In particular, if the canonical bilinear form (\ref{bilinear}) is everywhere non-degenerate, then the Jacobi curve $J_\alpha(t)$ is regular for each $\alpha$.
\end{prop}

\begin{proof}
Let $e^1(t),...,e^n(t)$ be given by Proposition \ref{Darboux}. Let $V^i_t$ be a time dependent vector field on $M$ such that $d\varphi_t(e^i(t))=V^i_t(\varphi_t(\alpha))$. It follows from the definition of the bilinear form (\ref{bilinear1}) and the invariance of the form $\omega$ under the flow $\varphi_t$ that
\[
\begin{split}
\left<\varphi_t^*V^i_t,\varphi_t^*V^i_t\right>^t&=\omega_\alpha(\varphi_t^*([\vec H,V^i_t]+\dot V^i_t),\varphi_t^*V^i_t)\\
&=\omega_{\varphi_t(\alpha)}([\vec H,V^i_t],V^i_t)\\
&=\left<V^i_t,V^i_t\right>_{\varphi_t(\alpha)}.
\end{split}
\]
\end{proof}

It is, therefore, natural to call a Hamiltonian vector field $\vec H$ regular if the corresponding canonical bilinear form is non-degenerate. In particular, if $\vec H$ is monotone, then it is regular.

\begin{defn}
Assuming that the Hamiltonian vector field $\vec H$ is regular. Let us denote by $J^o_\alpha(t)$ the derivative curve of the Jacobi curve $J_\alpha(t)$ at $\alpha$ defined in Definition \ref{der}. We define a Lagrangian distribution $\Lambda^o$ by
\[
\Lambda^o=J^o_\alpha(0).
\]
We will refer to distributions $\Lambda$ and $\Lambda^o$ as the \textit{vertical} and the \textit{horizontal} bundles, respectively. We will also refer to a tangent vector in the distributions $\Lambda$ and $\Lambda^o$ as a vertical vector and a horizontal vector, respectively. If $w$ is a tangent vector in $TM=\Lambda\oplus\Lambda^o$, then its components $w^v$ in $\Lambda$ and $w^h$ in $\Lambda^o$ are called vertical and horizontal parts of $w$, respectively.
\end{defn}

The Jacobi curve $J_\alpha$ and the derivative curve $J^o_\alpha$ satisfy the following property.

\begin{prop}\label{curves}
For each $\alpha$ in $M$, we have
\[
d\varphi_s(J_\alpha(t))=J_{\varphi_{s}(\alpha)}(t-s),\quad d\varphi_s(J^o_\alpha(t))=J^o_{\varphi_{s}(\alpha)}(t-s).
\]
\end{prop}

\begin{proof}
Let $J_\alpha(t)$ be the Jacobi curve at $\alpha$. It follows that
\[
d\varphi_s(J_\alpha(t))=d\varphi_s(d\varphi_t^{-1}\Lambda_{\varphi_t(\alpha)})=d\varphi_{t-s}^{-1}(\Lambda_{\varphi_{t-s}(\varphi_s(\alpha))})=J_{\varphi_s(\alpha)}(t-s).
\]
It also follows that $d\varphi_s(e^1(s+t)),...,d\varphi_s(e^n(s+t))$ is a canonical frame of $J_{\varphi_s(\alpha)}$. The second assertion follows from this.
\end{proof}

Similarly, we define the curvature operator of a Hamiltonian vector field by that of the Jacobi curves.

\begin{defn}
Assuming that the Hamiltonian vector field $\vec H$ is regular. Let $R_\alpha(t)$ be the curvature operators of the Jacobi curve $J_\alpha(t)$ at $\alpha$. The \textit{curvature operator} $\mathfrak R:\Lambda\to\Lambda$ of $\vec H$ is defined by
\[
\mathfrak R_\alpha=R_\alpha(0).
\]
\end{defn}

\begin{prop}\label{curvaturechar}
Assume that the Hamiltonian vector field $\vec H$ is regular. For each $\alpha$ in $M$ and each vector $v$ in $\Lambda_{\varphi_t(\alpha)}$, the following holds.
\[
R_{\alpha}(t)(d\varphi_t^{-1}(v))=d\varphi_t^{-1}(\mathfrak R_{\varphi_t(\alpha)}(v)).
\]
Moreover, for each vertical vector field $V$, the curvature operator $\mathfrak R$ satisfies
\[
\mathfrak R_\alpha(V)=-[\vec H,[\vec H,V]^h]^v(\alpha).
\]
\end{prop}

\begin{proof}
Let $e^1(t),...,e^n(t),f^1(t),...,f^n(t)$ be given by Proposition \ref{Darboux}. As observed in the proof of Proposition \ref{curves},
\[
t\mapsto(d\varphi_s(e^1(t+s)),... ,d\varphi_s(e^n(t+s)))
\]
is a canonical Darboux frame at $\varphi_s(\alpha)$. Therefore, we have
\[
\begin{split}
d\varphi_s(R_\alpha(t+s)\,e^i(t+s))&=-\frac{d^2}{dt^2}d\varphi_s(e^i(t+s))\\
&=R_{\varphi_s(\alpha)}(t)\,d\varphi_s(e^i(t+s)).
\end{split}
\]

If we set $t=0$, then we obtain
\[
d\varphi_s(R_\alpha(s)\,e^i(s))=R_{\varphi_s(\alpha)}(0)\,d\varphi_s(e^i(s))
\]
and the first assertion follows.

Let $V_t^i$ be a time-dependent vertical vector field on $M$ such that $d\varphi_t(e^i(t))=V_t^i(\varphi_t)$ in a neighborhood of a point $\alpha$ in the cotangent bundle $T^*M$. It follows from the definition of the canonical Darboux frame that
\[
f^i(t)=\dot e^i(t)=\varphi_t^*([\vec H,V_t^i]+\dot V_t^i)(\alpha)
\]
and
\[
R_\alpha(t)\,e^i(t)=-\varphi_t^*([\vec H,[\vec H,V_t^i]]+2[\vec H,\dot V_t^i]+\ddot V_t^i)(\alpha).
\]

Note that
\[
d\varphi_t(f^i(t))=[\vec H,V_t^i](\varphi_t(\alpha))+\dot V_t^i(\varphi_t(\alpha))
\]
is contained in the horizontal space. Therefore, $[\vec H,V_t^i]+\dot V_t^i$, and hence $[\vec H,\dot V_t^i]+\ddot V_t^i$, are horizontal vector fields. It follows that
\[
\begin{split}
\mathfrak R_\alpha(V_0^i(\alpha))&=-([\vec H,[\vec H,V_0^i]]+2[\vec H,\dot V_0^i]+\ddot V_0^i)(\alpha)\\
&=-([\vec H,[\vec H,V_0^i]+\dot V_0^i])^v(\alpha)\\
&=-[\vec H,[\vec H,V^i_0]^h]^v(\alpha).
\end{split}
\]
It remains to note that the maps $V\mapsto [\vec H,V]^h(\alpha)$ and $W\mapsto [\vec H,W]^v(\alpha)$ are tensorial on $\Lambda$ and $\Lambda^o$, respectively.
\end{proof}

\smallskip

\section{Reduction of curves in Lagrangian Grassmannians}\label{Reduction}

Let $v$ be a vector in a symplectic vector space $\mathcal V$. Let $v^\angle$ be the symplectic complement of $v$. Recall that the symplectic reduction $\tilde{\mathcal V}$ of $\mathcal V$ by $v$ is defined by
\[
\tilde{\mathcal V}= v^\angle/\Real\, v.
\]

The symplectic form $\omega$ descends to a symplectic form $\tilde\omega$ on $\tilde{\mathcal V}$. It follows that any Lagrangian subspace in $\mathcal V$ also descends to a Lagrangian subspace in $\tilde{\mathcal V}$. In particular, if $J$ is a curve in the Lagrangian Grassmannian $\mathcal L(\mathcal V)$, then it descends to a curve $\tilde J$ in $\mathcal L(\tilde{\mathcal V})$. Note also that the canonical bilinear form (\ref{bilinear1}) of the curve $J$ clearly descends to that of the curve $\tilde J$. It follows that $\tilde J$ is regular if $J$ is. Therefore, there is a curvature operator for the curve $\tilde J$ which is denoted by $\tilde R$. For the rest of this section, we recall how the curvature of $J$ relates to that of $\tilde J$. The reduced Jacobi curve was considered in \cite{AgChZe}. Here we give slightly different proofs of the results.

By Proposition \ref{structural}, we can find a canonical frame
\begin{equation}\label{reduce1}
\tilde e^1(t),...,\tilde e^{n-1}(t)
\end{equation}
and a curvature operator $\tilde R(t):\tilde J(t)\to \tilde J(t)$ satisfying
\[
\dot{\tilde e}^i(t)=\tilde f^i(t),\quad \dot{\tilde f}^i(t)=-\tilde R(t)\tilde e^i(t).
\]

Assume that $v$ is transversal to the $J(t)$ for all $t$. Since $J(t)$ is a Lagrangian subspace, $v^\angle$ and $J(t)$ intersect transversely. Therefore, there is a family of bases along $J(t)$, denoted by
\[
\bar e^1(t),...,\bar e^n(t),
\]
which is orthonormal with respect to the canonical bilinear form (\ref{bilinear1}) such that the first $n-1$ of them are contained in $v^\angle$ and they  descend to the canonical frame (\ref{reduce1}) of $\tilde J(t)$. Let $\Omega(t)$ be the matrix with $ij$-th entry defined by $\Omega_{ij}(t):=\omega(\dot{\bar e}^i(t),\dot{\bar e}^j(t))$ and let $U$ be the solution of the equation (\ref{UOmega}) with initial condition $U(0)=I$. Note that $\Omega_{ij}(t)=0$ if $i\neq n$ and $j\neq n$. Let $\bar\Omega(t)$ be the $n-1$-vector with $i$-th entry defined by $\bar\Omega_i(t):=\Omega_{ni}(t)$. The curvature of the curve $J$ and its reduction $\tilde J$ are related as follows.

\begin{prop}\label{reducedcurve}
Assume that $v$ is transversal to $J(t)$ for all $t$. Then
\[
U(t)^T\mathcal R(t)U(t)=\left(\begin{array}{ccc}
\tilde{\mathcal R}(t)-\frac{3}{4}\bar\Omega\otimes\bar\Omega & \frac{\dot c(t)}{c(t)}\bar\Omega(t)+\frac{1}{2}\dot{\bar\Omega}(t)\\
\frac{\dot c(t)}{c(t)}\bar\Omega(t)^T+\frac{1}{2}\dot{\bar\Omega}(t)^T & \frac{1}{4}|\bar\Omega(t)|^2-\frac{\ddot c(t)}{c(t)}
\end{array}\right),
\]
where $c(t)=\omega(v,\bar e^n(t))$.

Here $\mathcal R(t)$ and $\tilde{\mathcal R}(t)$ denote the matrix representations of $R(t)$ and $\tilde R(t)$, respectively. $\bar\Omega\otimes\bar\Omega$ is the matrix defined by
\[
\bar\Omega\otimes\bar\Omega(w)=\left<\bar\Omega,w\right>\bar\Omega.
\]
\end{prop}

\begin{proof}
Let $b_i(t)$ and $c_i(t)$ be functions defined by
\[
v=\sum_{i=1}^n(b_i(t)\bar e^i(t)+c_i(t)\dot{\bar e}^i(t)).
\]
By assumption, $\omega(v,\bar e^i(t))=0$ for all $i\neq n$. Therefore, by the condition
\[
\omega(\dot{\bar e}^i(t), \bar e^j(t))=\delta_{ij},
\]
we have $c_i\equiv 0$ for all $i\neq n$.

Since $\omega(v,\dot{\bar e}^i(t))=0$ for all $i\neq n$, we also have
\[
b_j=c_n\Omega_{nj}
\]
for all $j\neq n$.

Since $\omega(v,\bar e^n(t))=c_n(t)$, it also follows that $\dot c_n=b_n$. Therefore, we have
\begin{equation}\label{veq}
v=c_n(t)\left(\dot{\bar e}^n(t)+\sum_{j=1}^{n-1}\Omega_{nj}(t) \bar e^j(t)\right)-\dot c_n(t)\bar e^n(t).
\end{equation}

If we differentiate the above equation with respect to $t$, then we obtain
\begin{equation}\label{endd}
\begin{split}
& \ddot{\bar e}^n(t)=-\sum_{j=1}^{n-1}\left(\frac{\dot c_n(t)}{c_n(t)}\Omega_{nj}(t)+\dot\Omega_{nj}(t)\right)\bar e^j(t)\\
&\quad +\frac{\ddot c_n(t)}{c_n(t)}\bar e^n(t)-\sum_{j=1}^{n-1}\Omega_{nj}(t)\dot{\bar e}^j(t).
\end{split}
\end{equation}

On the other hand, since $\bar e^i(t)$ projects to $e^i(t)$, there is a function $a_i(t)$ such that
\[
\ddot{\bar e}^i(t)=-\sum_{j=1}^{n-1}\tilde{\mathcal R}_{ij}(t)\bar e^j(t)+a_i(t)v.
\]

By using  $\omega(v,\bar e^n(t))=c_n(t)$ again, we obtain $a_i=\frac{1}{c_n}\Omega_{ni}$. Therefore,
\begin{equation}\label{eidd}
\ddot{\bar e}^i(t)=-\sum_{j=1}^{n-1}\tilde{\mathcal R}_{ij}(t)\bar e^j(t)+\frac{1}{c_n(t)}\Omega_{ni}(t)v.
\end{equation}

Let $\bar E(t)=(\bar e^1(t),...,\bar e^n(t))^T$ and let $\bar F(t)=(\bar f^1(t),...,\bar f^n(t))^T$. Let $U$ be a solution of the equation $\dot U(t)=\frac{1}{2}U(t)\Omega(t)$ with $U(0)=I$ and let $E(t)=U(t)\bar E(t)$. By Lemma \ref{noncanonical}, $E(t)$ and $F(t):=\dot E(t)$ together form a Darboux basis. If we differentiate this twice, we obtain
\[
\begin{split}
&-\mathcal R(t)U(t)\bar E(t)\\
&=-\mathcal R(t)E(t)\\
&=\ddot E(t)\\
&=\ddot U(t)\bar E(t)+U(t)\Omega(t)\dot{\bar E}(t)+U(t)\ddot{\bar E}(t)\\
&=\frac{1}{2} \left(\frac{1}{2}U(t)\Omega(t)^2+U(t)\dot\Omega(t)\right)\bar E(t)+U(t)\Omega(t)\dot{\bar E}(t)+U(t)\ddot{\bar E}(t)
\end{split}
\]

Note that $\Omega$ is a skew-symmetric matrix satisfying $\Omega_{ij}=0$ if $i\neq n$ and $j\neq n$. If we combine the above equation with (\ref{veq}), (\ref{endd}), and (\ref{eidd}), then we obtain the following
\[
U(t)^T\mathcal R(t)U(t)=\left(\begin{array}{ccc}
\tilde{\mathcal R}(t)-\frac{3}{4}\bar\Omega\otimes\bar\Omega & \frac{\dot c_n(t)}{c_n(t)}\bar\Omega+\frac{1}{2}\dot{\bar\Omega}\\
\frac{\dot c_n(t)}{c_n(t)}\bar\Omega^T+\frac{1}{2}\dot{\bar\Omega}^T & \frac{1}{4}|\bar\Omega|^2-\frac{\ddot c_n(t)}{c_n(t)}
\end{array}\right),
\]
where $\bar\Omega$ is the vector in $\Real^n$ with $i$-th entry equal to $\Omega_{ni}$.
\end{proof}

Next, we define the reduced curvature operator of the Hamiltonian vector field $\vec H$. The reduction $\tilde J_\alpha$ of the curve $J_\alpha$ is defined by
\[
\tilde J_\alpha:=(J_\alpha\cap\vec H^\angle)/\Real\vec H.
\]

We also let $\tilde\Lambda_\alpha$ be the reduced distribution
\[
\tilde\Lambda_\alpha:=(\Lambda_\alpha\cap\vec H^\angle)/\Real\vec H.
\]

\begin{defn}
Assuming that the Hamiltonian vector field $\vec H$ is regular. Let $\tilde R_\alpha(t)$ be the curvature operators of the Jacobi curve $\tilde J_\alpha(t)$ at $\alpha$. The \textit{reduced curvature operator} $\tilde{\mathfrak R}:\tilde\Lambda\to\tilde\Lambda$ of $\vec H$ is defined by
\[
\tilde{\mathfrak R}_\alpha=\tilde R_\alpha(0).
\]
\end{defn}

In order to apply Proposition \ref{reducedcurve}, we need to assume that $\vec H$ is transversal to $\Lambda$ on a level set $\Sigma_c=H^{-1}(c)$. This condition is satisfied if $\vec H$ is monotone and $c$ is a regular value. Note that $\vec H^\angle_\alpha=\ker dH_\alpha=T_\alpha\Sigma_c$ if $c$ is regular value of $H$ and $\alpha$ is in $\Sigma_c:=H^{-1}(c)$.

\begin{prop}
Let $c$ be a regular value of $H$. Assume that the Hamiltonian vector field $\vec H(\alpha)$ is transversal to the space $\Lambda(\alpha)$ for all $\alpha$ in $\Sigma_c$. For $w$ in $\Lambda_\alpha\cap\vec H(\alpha)^\angle$, we have
\[
\left<\tilde{\mathfrak R}_\alpha(w),w\right>=\left<\mathfrak R_\alpha(w),w\right>+\frac{3}{4}\omega_\alpha([\vec H,[\vec H,\xi]],w)^2,
\]
where $\xi$ is a (local) section of $(\Lambda\cap\vec H^\angle)^\perp$ and $\perp$ denotes the orthogonal complement taken with respect to the canonical inner product of $\vec H$.
\end{prop}

\begin{proof}
By setting $t=0$ in the statement of Proposition \ref{reducedcurve}, we have
\begin{equation}\label{red}
\left<\tilde{\mathfrak R}_\alpha(w),w\right>=\left<\mathfrak R_\alpha(w),w\right>+\frac{3}{4}\left(\sum_{i=1}^{n-1}\Omega_{ni}(0)w_i\right)^2,
\end{equation}
where $w=\sum_{i=1}^nw_ie^i(0)$.

Using the notations of the proof of Proposition \ref{reducedcurve}, we have $\tilde e^n(t)=\varphi_t^*\xi(\alpha)$. It follows that
\[
\begin{split}
\Omega_{ni}(0)&=-\omega(\ddot{\tilde e}^n(0),\tilde e^i(0))\\
&=-\omega([\vec H,[\vec H,\xi]],\tilde e^i(0)).
\end{split}
\]

The result follows by combining this with (\ref{red}).
\end{proof}

\smallskip

\section{Existence of invariant distributions}\label{Existence}

In this section, we assume that a given Hamiltonian vector field $\vec H$ is monotone and it does not contain any conjugate point. Let $\varphi_t$ be the flow of a regular Hamiltonian vector field $\vec H$.

\begin{defn}
The point $\varphi_t(\alpha)$ is a conjugate point of $\alpha$ along  the flow $\varphi_t$ if $d\varphi_t(\Lambda_\alpha)$ and $\Lambda_{\varphi_t(\alpha)}$ do not intersect transversely for some $t>0$. Equivalently, $\varphi_t(\alpha)$ is a conjugate point if $J_\alpha(0)$ and $J_\alpha(t)$ do not intersect transversely for some $t> 0$.
\end{defn}

Under the above assumptions, we show that there are always two Lagrangian distributions $\Delta^\pm$ which are invariant under $\varphi_t$. Theorem \ref{intro1} also follows from the following.

\begin{thm}\label{Green}
Let $c$ be a regular value of the Hamiltonian $H$. Assume that the Hamiltonian vector field $\vec H$ is monotone and its flow does not contain any conjugate point on $\Sigma_c$. Then the following holds on $\Sigma_c$:
\begin{enumerate}
\item $\Delta^\pm_\alpha:=\lim_{t\to \pm\infty}J_\alpha(t)$ exists,
\item $\Delta^\pm$ are Lagrangian distributions which are invariant under $d\varphi_t$,
\item $\Delta^+\cap\Lambda=\Delta^-\cap\Lambda=\{0\}$,
\item $\vec H \subseteq \Delta^+\cap\Delta^-$,
\item $\Delta^\pm\subseteq\vec H^\angle=T\Sigma_c$.
\end{enumerate}
\end{thm}

Note that the above theorem does not require any compactness assumption on $\Sigma_c=H^{-1}(c)$. For the proof of Theorem \ref{Green}, it is convenient to introduce the reduction of $d\varphi_t$ which is also needed in the later sections. Let us consider the quotient bundle $\mathfrak V:=\vec H^\angle/\Real\vec H$. Both the symplectic structure $\omega$ and the flow $d\varphi_t$ descend to $\mathfrak V$. The descended objects are denoted by $\tilde\omega$ and $\widetilde{d \varphi}_t$, respectively. The bundle $\tilde\Lambda$ defined by $\tilde\Lambda:=(\Lambda\cap \vec H^\angle)/\Real\vec H$ is a Lagrangian sub-bundle of $\mathfrak V$.

Let $\tilde J_\alpha(t)$ be the reduced Jacobi curve defined by
\[
\tilde J_\alpha(t):=\widetilde {d\varphi}_t^{-1}(\tilde\Lambda_{\varphi_t(\alpha)})=(J_\alpha(t)\cap \vec H^\angle_\alpha)/\Real \vec H.
\]
The canonical frames of $\tilde J_\alpha$ are denoted by
\[
\tilde E_\alpha(t)=(\tilde e^1_\alpha(t),...,\tilde e^{n-1}_\alpha(t))^T,\quad \tilde F_\alpha(t)=(\tilde f^1_\alpha(t),...,\tilde f^{n-1}_\alpha(t))^T.
\]

Next, we adopt an argument in \cite[Proposition 1.16]{CoIt} and prove the following result which holds true for a general regular curve in the Lagrangian Grassmannian.

\begin{prop}\label{noconjugatepoint}
Let $J$ be a curve in the Lagrangian Grassmannian $\mathcal L(\mathcal V)$. Let $\Delta$ be a Lagrangian subspace of $\mathcal V$ such that $\Delta$ and $J(t)$ intersect transversely for all $t$. Let $v$ be a vector in $\Delta$. Assume that the curves $J(t)$ and $J(0)$ intersect transversely for all $t$. Then the same holds for the reduced curve $\tilde J$.
\end{prop}

\begin{proof}
Assume the contrary. Then there is a nonzero vector $w$ in $J(0)\cap v^\angle\cap (J(t_0)\oplus \Real v)$. Let $E(t)$ be a canonical frame and let $F(t)=\dot E(t)$. Let $D(t)$ be the matrix such that the components of
\[
-\dot D(t)^TE(t)+D(t)^TF(t)=-\dot D(0)^TE(0)+F(0)
\]
span $\Delta$ and $D(t)$ satisfies (\ref{2nd}) with $D(0)=I$.

Let $B$ be the matrix defined by
\begin{equation}\label{B}
B(t)=D(t)\int_0^tD(s)^{-1}(D(s)^T)^{-1}ds.
\end{equation}
$B$ is a solution of (\ref{2nd}) with initial conditions $B(0)=0$ and $\dot B(0)=I$.

Since $w$ is contained in $J(0)$, we can let $w=-a^TE(0)$ and get
\[
w=a^T(-\dot B(t)^TE(t)+B(t)^TF(t)).
\]

Since $w$ is contained in $J(t_0)\oplus\Real v$ and $v$ is transversal to the space $J(t_0)$, the $J^o(t_0)$-component of $v$ is given by the non-zero vector $a^TB(t_0)^TF(t_0)$. On the other hand, since $v$ is contained in $\Delta$, there is a vector $b$ such that
\[
v=b^T(-\dot D(t)^TE(t)+D(t)^TF(t)).
\]
It follows that $D(t_0)b=cB(t_0)a$ for some nonzero constant $c$. Note that $D(t_0)$ is invertible since $\Delta$ and $J(t_0)$ intersect transversely. Therefore, if we combine the above considerations with (\ref{B}), then
\[
a^Tb=ca^TD(t_0)^{-1}B(t_0)a=ca^T\left(\int_0^{t_0}D(s)^{-1}(D(s)^T)^{-1}ds\right)a\neq 0.
\]
However, since $w$ is contained in $v^\angle$, we also have
\[
\begin{split}
0&=\omega(v,w)\\
&=-\omega(b^TF(0),a^TE(0))\\
&=-b^Ta.
\end{split}
\]

This gives a contradiction.
\end{proof}

\begin{proof}[Proof of Theorem \ref{Green}]
We prove the statements for $\Delta^+$. That of $\Delta^-$ is similar and will be omitted. We will work with the reduced flow and find a Lagrangian sub-bundle $\tilde\Delta^+$ in $\mathfrak V$ which is invariant under $\widetilde{d\varphi}_t$ instead. It follows that the distribution $\Delta^+$ defined by
\begin{equation}\label{D+}
\Delta^+_\alpha:=\{v\in \vec H(\alpha)^\angle|v+\Real\vec H(\alpha)\in\tilde \Delta^+_\alpha\}
\end{equation}
is an invariant Lagrangian distribution.

Let $\tilde E_\alpha(t):=(\tilde e^1_\alpha(t),...,\tilde e^{n-1}_\alpha(t))^T$ be a canonical frame of the reduced curve $\tilde J_\alpha(t)$ at $\alpha$ and let $\tilde F_\alpha(t)=\dot{\tilde  E}_\alpha(t)$. Let $B(s,t)$ be the matrices defined by
\begin{equation}\label{E}
\tilde E_\alpha(t)=-B'(s,t)\tilde E_\alpha(s)+B(s,t)\tilde F_\alpha(s).
\end{equation}

By differentiating (\ref{E}) with respect to $t$, we obtain
\begin{equation}\label{F}
\tilde F_\alpha(t)=-\dot B'(s,t)\tilde E_\alpha(s)+\dot B(s,t)\tilde F_\alpha(s)
\end{equation}
and
\[
\begin{split}
&-\tilde{\mathcal R}_\alpha(t)B'(s,t)\tilde E_\alpha(s)+\tilde R_\alpha(t)B(s,t)\tilde F_\alpha(s)\\
&=\tilde{\mathcal R}_\alpha(t)\tilde E_\alpha(t)\\
&=\ddot B'(s,t)\tilde E_\alpha(s)-\ddot B(s,t)\tilde F_\alpha(s).
\end{split}
\]

It follows that
\[
\ddot B(s,t)=-\tilde{\mathcal R}_\alpha(t)B(s,t).
\]
Let $U_\alpha(s,t):=\dot B(s,t)B(s,t)^{-1}$. It satisfies
\[
\dot U_\alpha(s,t)+(U_\alpha(s,t))^2+\tilde{\mathcal R}_\alpha(t)=0.
\]

By the assumption of the theorem and Proposition \ref{noconjugatepoint}, $B(s,t)$ satisfies Assumption \ref{noconjugate}. It follows from Lemma \ref{Ubdd} that $U^+_\alpha(t)=\lim_{s\to\infty}U_\alpha(s,t)$ exists. Finally, we define
\begin{equation}\label{D+reduce}
\tilde\Delta_\alpha^+:=\textbf{span}\{\tilde F_\alpha(0)-U_\alpha^{+}(0)\tilde E_\alpha(0)\}.
\end{equation}

If we set $t=0$ in (\ref{E}) and (\ref{F}), then we obtain
\[
\tilde J_\alpha(s)=\textbf{span}\{\tilde F_\alpha(0)-U_\alpha(s,0)\tilde E_\alpha(0)\}.
\]
Therefore, we have $\lim_{s\to\infty}\tilde J_\alpha(s)=\tilde\Delta_\alpha^+$ and (1) follows. It also follows from (\ref{D+reduce}) that $\tilde\Delta^+\cap\tilde\Lambda=\{\Real\vec H\}$. Since $\vec H$ is not contained in $\Lambda$, (3) follows. (4) follows from (\ref{D+}) and (5) follows from taking skew-orthogonal complement in (4). Finally, by Proposition \ref{curves},
\[
d\varphi_s(\tilde J_\alpha(t))=\tilde J_{\varphi_s(\alpha)}(t-s).
\]
If we let $t\to\infty$, then we see that $\tilde\Delta^+$ is invariant under $\widetilde{d\varphi}_t$. Since $\vec H$ is also invariant under $d\varphi_t$, (2) follows.
\end{proof}

\smallskip

\section{Rigidity of the reduced curvature}\label{Rigidity}

In this section, we will give the proof of Theorem \ref{totalreduce}. In fact, Theorem \ref{totalreduce} is an immediate consequence of the following result and Theorem \ref{intro1}.

\begin{thm}
Let $c$ be a regular value of the Hamiltonian $H$ and let $\Sigma_c:=H^{-1}(c)$. Assume that the Hamiltonian vector field $\vec H$ is regular and its flow $\varphi_t$ preserves a Lagrangian distribution on $\Sigma_c$ which is everywhere transversal to $\Lambda$. Then the trace $\mathfrak{\tilde r}$ of the reduced curvature $\mathfrak{\tilde R}$ satisfies
\[
\int_{\Sigma_c}\mathfrak{\tilde r}_\alpha d\mu(\alpha)\leq 0,
\]
where $\mu$ is any invariant measure defined on $\Sigma_c$. Moreover, equality holds only if $\tilde{\mathfrak r}=0$ on the support of $\mu$.
\end{thm}

\begin{proof}
Let $\Delta$ be the Lagrangian distribution defined on $\Sigma_c$ which is invariant under the flow $\varphi_t$. Let $\tilde \Delta$ be defined by $\tilde\Delta:=(\Delta\cap\vec H^\angle)/\Real\vec H$. Then $\tilde\Delta$ is a sub-bundle of $\mathfrak V$ which is invariant under $\widetilde{d\varphi}_t$. Let $E_\alpha(t)=(e^1_\alpha(t),...,e^{n-1}_\alpha(t))^T$ be a canonical frame at $\alpha$, let $F_\alpha(t)=\dot E_\alpha(t)$. Since $\tilde\Lambda$ and $\tilde\Delta$ intersect transversely, we can let $S_0$ be the matrix such that
\[
F_\alpha(0)+S_0E_\alpha(0)
\]
span the space $\tilde\Delta_\alpha$.

It follows that
\begin{equation}\label{EF}
F_\alpha(0)-S_0E_\alpha(0)=B_\alpha(t)^TF_\alpha(t)-\dot B_\alpha(t)^TE_\alpha(t)
\end{equation}
where $B_\alpha(t)$ is a solution of (\ref{2nd}) satisfying the initial conditions $B_\alpha(0)=I$ and $\dot B_\alpha(0)=S_0$ and $\mathcal R=\tilde{\mathcal R}_\alpha$ is the curvature of the reduced Jacobi curve $\tilde J_\alpha(t)$.

Note that $d\varphi_t(e_i(t))$ is vertical and $d\varphi_t(f_i(t))$ is horizontal. Since the components of (\ref{EF}) span $\tilde \Delta$, $\tilde\Delta$ is invariant under $\widetilde{d\varphi_t}$, and $\tilde\Delta$ is transversal to $\tilde\Lambda$, the matrix $B_\alpha (t)$ is invertible for all $t$ and all $\alpha$.

Let $S_\alpha(t)=\dot B_\alpha(t)B_\alpha(t)^{-1}$. Then $S_\alpha$ is the solution of (\ref{S}) which satisfies the initial condition $S_\alpha(0)=S_0$. It follows that the trace $\textbf{tr}(S_\alpha(t))$ of $S_\alpha(t)$ satisfies the following equation
\begin{equation}\label{smallRiccati}
\textbf{tr}(\dot S_\alpha(t))+\textbf{tr}(S_\alpha(t)^2)+\textbf{tr}(\mathcal{\tilde R}_\alpha(t))=0.
\end{equation}

By integrating (\ref{smallRiccati}) with respect to time $t$, we obtain
\begin{equation}\label{integRiccati}
\textbf{tr}(S_\alpha(1))-\textbf{tr}(S_\alpha(0))+\int_0^1\textbf{tr}(S_\alpha(t)^2)dt+\int_0^1\textbf{tr}(\mathcal{\tilde R}_\alpha(t))dt=0.
\end{equation}

Since $\textbf{tr}(S_\alpha(t))$ is independent of the choice of frames $\tilde E_\alpha(t)$, it defines a function $\alpha\mapsto \textbf{tr}(S_\alpha(t))$. Moreover, we have $\textbf{tr}(S_\alpha(t))=\textbf{tr}(S_{\varphi_t(\alpha)}(0))$. Therefore, (\ref{integRiccati}) becomes
\begin{equation}\label{integRiccati2}
\textbf{tr}(S_{\varphi_1(\alpha)})-\textbf{tr}(S_\alpha(0))+\int_0^1\textbf{tr}(S_{\varphi_t(\alpha)}(0)^2)dt+\int_0^1\tilde{\mathfrak r}_{\varphi_t(\alpha)}dt=0.
\end{equation}

If we integrate (\ref{integRiccati2}) with respect to the invariant measure $\mu$, then we obtain
\begin{equation}\label{integRiccati3}
\int_{\Sigma_c}\textbf{tr}(S_\alpha (0)^2)+\tilde{\mathfrak r}_\alpha d\mu(\alpha)=0.
\end{equation}

It follows that $\int_{\Sigma_c}\tilde{\mathfrak r}_{\alpha}d\mu(\alpha)\leq 0$. Moreover, equality holds only if $S_\alpha(t)= 0$ for $\mu$-almost all $\alpha$. Since $t\mapsto S_\alpha(t)$ is smooth, there is a set of full $\mu$-measure $\mathcal O$ in $M$ such that $S_\alpha(t)= 0$ for all $t$ in $[0,1]$ and for each $\alpha$ in $\mathcal O$.

Finally, it follows from (\ref{smallRiccati}) and the smoothness of $\tilde{\mathfrak r}$ that $\tilde{\mathfrak r}=0$ on the support of $\mu$.
\end{proof}

\smallskip

\section{Hyperbolicity under negative reduced curvature}\label{Hyperbolicity}

In this section, we show that the Hamiltonian flow of a monotone Hamiltonian vector field is Anosov if the reduced curvature is bounded above and below by negative constants. First, we show that if the reduced curvature is everywhere non-positive, then the Hamiltonian flow has no conjugate point. A proof of this can be found in \cite{Ro}. We supply the proof here for completeness. From here on, unless otherwise stated, we endow the manifold $M$ with a Riemannian metric denoted by $\left<\cdot,\cdot\right>$. It is defined by the condition that the canonical frame $E(0)$ and $F(0)$ of the Jacobi curve is orthonormal. The corresponding norm is denoted by $|\cdot|$. Similarly, we also endow the vector bundle $\mathfrak V$ with a Riemannian metric denoted using the same symbol $\left<\cdot,\cdot\right>$ such that the canonical frame $\tilde E(0)$ and $\tilde F(0)$ of the reduced Jacobi curve is orthonormal.

\begin{thm}
Assume that the reduced curvature of a regular Hamiltonian vector field $\vec H$ is non-positive. Then, for each $\tilde w$ in $\tilde\Lambda$, $|\widetilde{d\varphi}_t(\tilde w)^h|$ is increasing for all $t>0$ and decreasing for all $t<0$. In particular, the flow of $\vec H$ has no conjugate point.
\end{thm}

\begin{proof}
We will only do the case $t>0$. Let $\tilde E_\alpha(t)=(\tilde e^1_\alpha(t),...,\tilde e^{n-1}_\alpha(t))^T$ be a canonical frame of the Jacobi curve at $\alpha$. Let $B(t)$ be the solution of (\ref{2nd}) with initial conditions $B(0)=0$ and $\dot B(0)=I$. Then
\[
\tilde E_\alpha(0)=\dot B(t)^T\tilde E_\alpha(t)-B(t)^T\tilde F_\alpha(t).
\]

In other words, if we define $S(t)=\dot B(t)B(t)^{-1}$, then $S(t)$ is a solution of the matrix Riccati equation (\ref{S}) which is defined wherever $B(t)$ is invertible. Since $\frac{1}{t}I$ is also a solution of (\ref{S}) with $R(t)\equiv 0$, it follows from Theorem \ref{comparison} that $S(t)\geq \frac{1}{t}I$ for all $t>0$. It also follows from Theorem \ref{comparison} that $S(t)$ is bounded above by the solutions of the equation
\[
\dot S(t)+\mathcal R_\alpha(t)=0.
\]
It follows that $S(t)$ is defined for all $t>0$ and $B(t)$ is invertible. Therefore, by Proposition \ref{noconjugatepoint}, there is no point conjugate to $\alpha$ along $\varphi_t$.

Moreover, if we let $\tilde w=b^T\tilde E_\alpha(0)$, then
\[
\begin{split}
\frac{d}{dt}|(\widetilde{d\varphi}_t(\tilde w))^h|^2&=2b^TB(t)^T\dot B(t)b\\
&=2b^TB(t)^TS(t)B(t)b> 0
\end{split}
\]
for all $t>0$.
\end{proof}

\begin{thm}\label{negativecur}
Assume that there are positive constants $k$ and $K$ such that the reduced curvature $\tilde{\mathfrak R}$ satisfies $-K^2I\geq \tilde{\mathfrak R}\geq -k^2I$ on $\Sigma_c:=H^{-1}(c)$, where $c$ is a regular value of $H$. Then there is a Riemannian inner product and invariant distributions $\Delta^s$ and $\Delta^u$ defined on $\bigcup_{\alpha\in\Sigma_c}\vec H^\angle(\alpha)$ satisfying the followings:
\begin{enumerate}
\item $\vec H^\angle=\textbf{span}\{\vec H\}\oplus\Delta^u\oplus\Delta^s$,
\item $\Delta^+=\textbf{span}\{\vec H\}\oplus\Delta^s$,
\item $\Delta^-=\textbf{span}\{\vec H\}\oplus\Delta^u$,
\item there is a constant $C>0$ such that $|d\varphi_{t}(w)|\leq Ce^{-Kt}|w|$ for all $t\geq 0$ and for all $w$ in $\Delta^s$,
\item  $|d\varphi_{-t}(w)|\leq Ce^{-Kt}|w|$ for all $t\geq 0$ and for all $w$ in $\Delta^u$.
\end{enumerate}

In particular, the flow $\varphi_t$ is Anosov on $\Sigma_c$.
\end{thm}

\begin{proof}
We use the notations in the proof of Theorem \ref{Green}. Let $\dot D^+(t)=U^+(t)D^+(t)$ with $D^+(0)=I$. If $\tilde w$ is a vector in $\tilde\Delta^+$, then there is a vector $b$ such that
\[
\tilde w=b^T(-\dot D^+(t)^T\tilde E_\alpha(t)+D^+(t)^T\tilde F_\alpha(t)).
\]

We extend the canonical inner product defined on $\tilde\Lambda$ to an inner product, still denoted by  $\left<\cdot,\cdot\right>$, of the bundle $\mathfrak V$ such that the basis $\tilde e^1_\alpha(0),..., \tilde e^{n-1}_\alpha(0),\tilde f^1_\alpha(0),...,\tilde f^{n-1}_\alpha(0)$ is orthonormal. It follows that
\begin{equation}\label{inD}
|\widetilde{d\varphi}_t(\tilde w)|^2=|D^+(t)b|^2+|U^+(t)D^+(t)b|^2.
\end{equation}

By Lemma \ref{nonpositiveRiccati}, we have
\[
U^+\leq-kI \quad\text{ and }\quad |D^+(t)b|^2\leq |b|^2e^{-2Kt}.
\]

By combining this with (\ref{inD}), we obtain
\[
\begin{split}
|\widetilde{d\varphi}_t(\tilde w)|^2&\leq \left(1+k^2\right)\,|D^+(t)b|^2\\
&\leq \left(1+k^2\right)\,|b|^2e^{-2Kt}\\
&\leq \frac{1+k^2}{1+K^2}\,|\tilde w|^2e^{-2Kt}.
\end{split}
\]

The rest follows from \cite[Proposition 5.1]{Wo} and the definition of $\Delta^+$ in the proof of Theorem \ref{Green}.
\end{proof}

\smallskip

\section{On the invariant bundles of the reduced flow}\label{OnThe}

Let $\tilde J_\alpha$ be the reduced Jacobi curve of $J_\alpha$. The reduced Jacobi curve and the derivative curve $\tilde J^o_\alpha$ give a splitting of the bundle $\mathfrak V=\tilde J_\alpha(0)\oplus\tilde J_\alpha^o(0)$. Let $\tilde v$ be an element in $\mathfrak V$. The $\tilde J_\alpha(0)$- and the $\tilde J_\alpha^o(0)$-components of $\tilde v$ are denoted by $\tilde v^v$ and $\tilde v^h$ respectively.

In this section, we prove the following characterization of the invariant bundles $\tilde\Lambda^\pm$ defined in the proof of Theorem \ref{Green}.

\begin{thm}\label{deltachar}
Assume that the Hamiltonian vector field $\vec H$ is monotone. Assume that $\Sigma_c:=H^{-1}(c)$ is compact and the flow of $\vec H$ has no conjugate point on $\Sigma_c$. Suppose that there is no vector $\tilde w$ in $\mathfrak V$ such that $|\widetilde{d\varphi}_t(\tilde w)^h|$ is bounded for all $t$. Then
\[
\tilde\Delta^\pm=\left\{\tilde w \Big| \sup_{\pm t\geq 0}|\widetilde{d\varphi}_t(\tilde w)^h|<+\infty\right\}.
\]
\end{thm}

In particular, the above theorem applies when the flow of $\vec H$ is Anosov on $\Sigma_c$.

\begin{lem}\label{bddDelta}
Let $c$ be a regular value of $H$. Assume that the Hamiltonian vector field $\vec H$ is monotone and its flow has no conjugate point on $\Sigma_c$.
Assume that the reduced curvature $\tilde{\mathfrak R}$ of $\vec H$ satisfies $\tilde{\mathfrak R}\geq -k^2I$ on $\Sigma_c$. Let $\tilde v$ be in $\mathfrak V_\alpha$ with $\alpha$ contained in $\Sigma_c$ and such that $|\widetilde{d\varphi}_t(\tilde v)^h|$ is uniformly bounded for all $t>0$ (resp. $t<0$). Then $\tilde v$ is contained in $\tilde\Delta^+_\alpha$ (resp. $\tilde\Delta^-_\alpha$).
\end{lem}

\begin{proof}
We will only prove the statement for $\tilde\Delta^+$. The one for $\tilde\Delta^-$, being very similar, will be omitted. Let $\tilde v$ be a tangent vector in $\mathfrak V_\alpha$ such that $t\mapsto \widetilde{d\varphi}_t(\tilde v)$ is uniformly bounded for all $t>0$. Since the flow of $\vec H$ has no conjugate point, there is a vector $\tilde v_t$ in $J_\alpha(t)$ such that the horizontal components of $\tilde v$ and $\tilde v_t$ are the same. It follows that $\tilde v-\tilde v_t$ is vertical for each $t$.

Let $\tilde E(t)=(\tilde e^1_\alpha(t),...,\tilde e^{n-1}_\alpha(t))^T$ be canonical frame and let $\tilde F(t)=\dot{\tilde E}(t)$. Let $B(s)$ be the solution of (\ref{2nd}) with initial conditions $B(0)=0$ and $B'(0)=I$. Let $b(t)$ be a family of vectors in $\Real^n$ defined by $\tilde v-\tilde v_t=b(t)^T\tilde E(0)$. Then we have
\[
\tilde v-\tilde v_t=b(t)^TB'(s)^T\tilde E(s)-b(t)^TB(s)^T\tilde F(s).
\]

By assumption, there is a constant $K>0$ such that $|B(t)b(t)|\leq K$ for all $t>0$. By Lemma \ref{Jacobiblowup}, there is $T_n>0$ such that
\[
\frac{K}{|b(t)|}\geq\frac{|B(t)b(t)|}{|b(t)|}\geq n
\]
for all $t>T_n$.

Therefore, $\lim_{t\to\infty}b(t)=0$ and $\lim_{t\to\infty}\tilde v_t=\tilde v$. Since $\tilde v_t$ is contained in $\tilde J_\alpha(t)$ for all $t>0$, $\tilde v$ is contained in $\tilde \Delta^+_\alpha$ as claimed.
\end{proof}

\begin{lem}\label{stbdd}
Suppose that the assumptions of Theorem \ref{deltachar} are satisfied. Then for each $s_0>0$ (resp. $s_0<0$), there is a constant $C>0$ such that
\[
|\widetilde{d\varphi}_t(\tilde w)^h|\geq C |\widetilde{d\varphi}_s(\tilde w)^h|
\]
for all $\tilde w$ in $\tilde \Lambda$ and for all $t\geq s\geq s_0$ (resp. $t\leq s\leq s_0$).
\end{lem}

\begin{proof}
Suppose that the conclusion does not hold. Then there are vectors $\tilde w_n$ in $\tilde\Lambda$ and numbers $t_n\geq s_n\geq s_0$ such that
\[
|\widetilde{d\varphi}_{t_n}(\tilde w_n)^h|< \frac{1}{n} |\widetilde{d\varphi}_{s_n}(\tilde w_n)^h|.
\]

By multiplying $\tilde w_n$ by a constant, we can assume that $|\tilde w_n|=1$. By compactness, we can assume that $\tilde w_n$ converges to $\tilde w$ in $\tilde \Lambda$. Let $u_n$ be the number which achieves the maximum of $|\widetilde{d\varphi}_t(\tilde w_n)^h|$ over $t$ in $[0,t_n]$. It follows that
\[
|\widetilde{d\varphi}_{u_n}(\tilde w_n)^h| \geq |\widetilde{d\varphi}_{s_0}(\tilde w_n)^h|
\]
is bounded below by a positive constant uniformly in $n$ since $\tilde w_n$ is convergent. Therefore, $u_n$ is also bounded below by a positive constant uniformly in $n$.

Let $\tilde v_n=\frac{\widetilde{d\varphi}_{u_n}(\tilde w_n)}{|(\widetilde{d\varphi}_{u_n}(\tilde w_n))^h|}$ and let $a_n$ be vectors defined by
\[
\tilde w_n=a_n^T(\dot B(t)^TE(t)-B(t)^TF(t))
\]
where $B$ is a solution of (\ref{2nd}) with initial conditions $B(0)=0$ and $\dot B(0)=I$.

Let $S(t)=B(t)^{-1}\dot B(t)$. Then $S(t)$ satisfies (\ref{S}). By Lemma \ref{Riccatibdd}, it follows that $\tilde v_n$ satisfies
\[
\begin{split}
|\tilde v_n|&\leq 1+\frac{|\widetilde{d\varphi}_{u_n}(\tilde w_n)^v|}{|\widetilde{d\varphi}_{u_n}(\tilde w_n)^h|}\\
&= 1+\frac{|\dot B(u_n)a_n|}{|B(u_n)a_n|}\\
&\leq 1+k\coth(ku_n).
\end{split}
\]

Since $u_n$ is bounded uniformly from below by a positive constant, $|\tilde v_n|$ is also bounded uniformly and we can assume that $\tilde v_n$ converges to a vector $\tilde v$. By the definition of $u_n$, we have
\begin{equation}\label{bddJacobi}
|\widetilde{d\varphi}_t(\tilde v_n)^h|=\frac{\left|\widetilde{d\varphi}_{t+u_n}(\tilde w_n)^h\right|}{\left|\widetilde{d\varphi}_{u_n}(\tilde w_n)^h\right|}\leq 1
\end{equation}
for $-u_n\leq t\leq t_n-u_n$. By assumption, $\widetilde{d\varphi}_{-u_n}(\tilde v_n)$ is contained in $\tilde\Lambda$ and
\[
|\widetilde{d\varphi}_{t_n-u_n}(\tilde v_n)^h|= \frac{\left|\widetilde{d\varphi}_{t_n}(\tilde w_n)^h\right|}{\left|\widetilde{d\varphi}_{u_n}(\tilde w_n)^h\right|}\leq \frac{\left|\widetilde{d\varphi}_{t_n}(\tilde w_n)^h\right|}{\left|\widetilde{d\varphi}_{s_n}(\tilde w_n)^h\right|}< \frac{1}{n}.
\]

If both $u_n$ and $t_n-u_n$ have convergent subsequence, then it violates the assumption that there is no vector $\tilde w$ in $\mathfrak V$ such that $\left|\widetilde{d\varphi}_t(\tilde w_n)^h\right|$ is bounded for all $t$. If both $-u_n\to -\infty$ and $t_n-u_n\to +\infty$. Then this violates the assumption that there is no bounded reduced non-zero Jacobi field. If one of $-u_n$ or $t_n-u_n$ has a convergent subsequence, then one of $\widetilde{d\varphi}_{-u_n}(\tilde v_n)$ or $\widetilde{d\varphi}_{t_n-u_n}(\tilde v_n)$ converges to a vector in $\tilde\Lambda$. This vector is also contained in either $\tilde \Delta^+$ or $\tilde \Delta^-$ by Lemma \ref{bddDelta} and (\ref{bddJacobi}). This violates (3) of Theorem \ref{Green}.
\end{proof}

\begin{proof}[Proof of Theorem \ref{deltachar}]
One inclusion follows from Lemma \ref{bddDelta}. For the other inclusion, let $\tilde w$ be in $\tilde \Delta^+_\alpha$. Let $\tilde w_\tau$ be the vector in $\tilde J_\alpha(\tau)$ such that $\tilde w^h=\tilde w_\tau^h$. By the definition of $\tilde\Delta^+$, we have $\lim_{\tau\to\infty}\tilde w_\tau=w$. Fix $s_0<0$. By Lemma \ref{stbdd}, there is a constant $C>0$ such that
\[
|\widetilde{d\varphi}_t(\tilde u)^h|\geq C|\widetilde{d\varphi}_s(\tilde u)^h|
\]
for all $t\leq s\leq s_0$ and for all $\tilde u$ in $\tilde\Lambda$.

Let $\tilde u=\widetilde{d\varphi}_{\tau}(\tilde w_\tau)$, $t=-\tau$, and $s=-\tau+\epsilon$. Then we obtain
\[
|\tilde w_\tau^h|\geq C|\widetilde{d\varphi}_{\epsilon}(\tilde w_\tau)^h|.
\]

By letting $\tau$ goes to $+\infty$, we obtain
\begin{equation}\label{Jacobibdd}
|\tilde w^h|\geq C|\widetilde{d\varphi}_{\epsilon}(\tilde w)^h|.
\end{equation}

Therefore, $|\widetilde{d\varphi}_{\epsilon}(\tilde w)^h|<+\infty$ for all $\epsilon\geq 0$.
\end{proof}

\smallskip

\section{Monotone Anosov Hamiltonian flows without conjugate point}\label{MonotoneAnosov}

In this section, we give various equivalent conditions which guarantee that a monotone Hamiltonian vector field without conjugate point is Anosov. More precisely, we will prove the following.

\begin{thm}\label{main1}
Let $\vec H$ be a monotone Hamiltonian vector field without conjugate point. Assume that $\Sigma_c=H^{-1}(c)$ is compact. Then the followings are equivalent.
\begin{enumerate}
\item $\tilde \Delta^+\cap\tilde\Delta^-=\{0\}$,
\item $\tilde\Lambda=\tilde\Delta^+\oplus\tilde\Delta^-$,
\item there is no vector $\tilde w$ in $\mathfrak V$ such that $|\widetilde{d\varphi}_t(\tilde w)^h|$ is bounded uniformly in $t$,
\item there are constants $c_1, c_2>0$ such that
\[
|\widetilde{d\varphi}_{\pm t}(\tilde w)|\leq c_1|\tilde w|e^{-c_2 t}
\]
for all $t\geq 0$ and $\tilde w$ in $\tilde\Delta^\pm$.
\end{enumerate}
\end{thm}

\begin{lem}\label{decay}
Under the assumptions of Theorem \ref{deltachar},
\[
\lim_{t\to\pm\infty}\sup_{|\tilde w|=1,\tilde w\in\tilde\Delta^\pm}|\widetilde{d\varphi}_t(\tilde w)|=0.
\]
\end{lem}

\begin{proof}
Suppose the statement for $\tilde\Delta^+$ does not hold. Then there is $\epsilon>0$, a sequence $t_n>0$ going to $\infty$, and a sequence $\tilde w_n$ in $\tilde \Delta^+$ satisfying $|\tilde w_n|=1$ such that
\[
|\widetilde{d\varphi}_{t_n}(\tilde w_n)|>\epsilon.
\]

Since $\widetilde{d\varphi}_{t_n}(\tilde w_n)$ is contained in $\tilde\Delta^+$, $|\widetilde{d\varphi}_{t_n}(\tilde w_n)|$ is uniformly bounded in $n$ by compactness and Theorem \ref{deltachar}. Therefore, $\widetilde{d\varphi}_{t_n}(\tilde w_n)$ converges to $\tilde w\neq 0$. Since $\widetilde{d\varphi}_{t_n}(\tilde w_n)$ is contained in $\tilde\Delta^+$, $|\widetilde{d\varphi}_{t+t_n}(\tilde w_n)|$ is uniformly bounded for all $n$ and $t\geq -t_n$ by (\ref{Jacobibdd}). Hence, by letting $n\to\infty$,  $|\widetilde{d\varphi}_t(\tilde w)|$ is uniformly bounded in $t$. This contradicts the assumption of the lemma.
\end{proof}

\begin{lem}\label{anosovaction}
Let $\vec H$ be monotone and without conjugate point. Let $c$ be a regular value of $H$ and assume that $\Sigma_c=H^{-1}(c)$ is compact. Then there is no vector $\tilde w$ in $\mathfrak V$ such that $|\widetilde{d\varphi}_t(\tilde w)^h|$ is bounded for all $t$ if and only if there are constants $c_1, c_2>0$ such that
\begin{equation}\label{anosovproperty}
|\widetilde{d\varphi}_{\pm t}(\tilde w)|\leq c_1|\tilde w|e^{-c_2 t}
\end{equation}
for all $t\geq 0$ and $\tilde w$ in $\tilde\Delta^\pm$.
\end{lem}

\begin{proof}
Clearly, (\ref{anosovproperty}) implies that $|\widetilde{d\varphi}_t(\tilde w)^h|$ is not bounded for all $t$. Conversely, let
\[
\phi^+(t)=\sup_{|\tilde w|=1,\tilde w\in\tilde\Delta^+}|\widetilde{d\varphi}_t(\tilde w)|.
\]
Then $\phi^+$ is uniformly bounded for all $t\geq 0$ (see (\ref{Jacobibdd})), $\phi^+(t+s)\leq\phi^+(s)\phi^+(t)$ for all $s, t\geq 0$, and $\lim_{s\to\infty}\phi^+(s)=0$ (Lemma \ref{decay}). The rest follows from \cite[Lemma 3.12]{Eb}.
\end{proof}

\begin{proof}[Proof of Theorem \ref{main1}]
By a count in dimensions, (1) and (2) are equivalent. By Lemma \ref{bddDelta}, (1) implies (3).  By Theorem \ref{deltachar}, (3) implies (1). (3) and (4) are equivalent by Lemma \ref{anosovaction}.
\end{proof}

\begin{proof}[Proof of Theorem \ref{main2}]
By Theorem \ref{main1}, it is enough to show that (3) of Theorem \ref{main1} is equivalent to (1) of Theorem \ref{main2}. This, in turn, follows from  \cite[Proposition 5.1]{Wo}.
\end{proof}

\smallskip

\section{The case with non-positive reduced curvature}\label{The}

In this section, we give the proof of Theorem \ref{nonpositive}. Under the assumption that the reduced curvature of $\vec H$ is non-positive, the following is a characterization of when the flow of $\vec H$ is Anosov.

\begin{lem}\label{anosovJo}
Let $c$ be a regular value of $H$ and assume that $\Sigma_c=H^{-1}(c)$ is compact. Assume that, for each  $\tilde v$ in $\tilde\Lambda_\alpha$ with $\alpha$ in $\Sigma_c$, $|\widetilde{d\varphi}_t(\tilde v)^h|$ is increasing for each $t>0$ and decreasing for each $t<0$. Then the followings are equivalent.
\begin{enumerate}
\item the Hamiltonian flow is Anosov on $\Sigma_c$,
\item $\bigcap_{t\in\Real}\tilde J^o_\alpha(t)=\emptyset$ for each $\alpha$ in $\Sigma_c$.
\end{enumerate}
In particular, the above conditions are equivalent if the reduced curvature of the Hamiltonian is non-positive.
\end{lem}

\begin{proof}
Let us fix a vector $\tilde w$ and let $b(t)$ be defined by
\[
\tilde w=-\dot b(t)^TE(t)+b(t)^TF(t).
\]

First, assume that $\tilde w$ is contained in $\bigcap_{t\in\Real}\tilde J^o(t)$. By assumption, we have $\dot b\equiv 0$. Therefore, $b(t)$ is constant independent of $t$. It follows that $|\widetilde{d\varphi}_t(\tilde w)^h|$ is constant and the Hamiltonian flow is not Anosov by Theorem \ref{main1}.

Conversely, by assumption and (\ref{Jacobibdd}), $|\widetilde{d\varphi}_{\pm t}(\tilde w)^h|\leq |\tilde w^h|$ for all $t\geq 0$ and for all $\tilde w$ in $\tilde\Delta^\pm$. Since $\tilde\Delta^\pm$ is invariant, we have $|\widetilde{d\varphi}_{\pm t+s}(\tilde w)^h|\leq |\widetilde{d\varphi}_{s}(\tilde w)^h|$ for all $s$. Therefore, if $\tilde w$ is in $\tilde\Delta^+\cap\tilde\Delta^-$, then it follows that $t\mapsto|\widetilde{d\varphi}_{\pm t}(\tilde w)^h|$ is both non-increasing and non-decreasing. Therefore, $t\mapsto|\widetilde{d\varphi}_{t}(\tilde w)^h|$ is constant in $t$.

Let $U^+$ be as in Theorem \ref{Green} and let $D^+$ be defined by $\dot D^+(t)=U^+(t)D^+(t)$ with initial condition $D^+(0)=I$. It follows that
\[
0\geq\frac{d}{dt}|D^+(t)\tilde b|^2=2\left<U^+(t)D^+(t)\tilde b,D^+(t)\tilde b\right>
\]
for all $t>0$ and for all vector $\tilde b$. Since $D^+$ is invertible, $U^+\leq 0$.

Let $b$ be a vector in $\Real^n$ such that
\[
\tilde w=b^T(D^{+}(t)^TF(t)-\dot D^+(t)^TE(t)).
\]
It follows that $|\widetilde{d\varphi}_{t}(\tilde w)^h|=|D^{+}(t)b|$ is constant and we have
\[
0= \frac{1}{2}\frac{d}{dt}\left(b^TD^{+}(t)^TD^{+}(t)b\right)=b^TD^{+}(t)^TU^+(t)D^{+}(t)b.
\]

Since $U^+\leq 0$, we have $\dot D^+(t)^Tb=U^+(t)D^+(t)^Tb=0$. Since $D^+(t)^Tb=b$, we have $\tilde w=b^TF(t)$. This shows (2) implies (1).
\end{proof}

\begin{prop}\label{nonpos1}
Let $c$ be a regular value of $H$ and assume that $\Sigma_c=H^{-1}(c)$ is compact.
Assume that the Hamiltonian flow has no conjugate point. Fix a vector $b$. If $\tilde{\mathcal R}_\alpha(t)b^T\tilde E(t)\geq 0$ for all $t$, then $\tilde {\mathcal R}_\alpha(t)b^T\tilde E(t)= 0$ for all $t$ and $b^T\tilde F(0)$ is contained in $\bigcap_{t\in\Real}\tilde J^o(t)$.
\end{prop}

\begin{proof}
Let $u(t)=b^TU^+(t)b$. Then
\[
\dot u(t)+u(t)^2+r(t)=0
\]
where $r(t)=b^T\tilde{\mathcal R}_\alpha(t)b+b^TU^+(t)^2b-(b^TU^+(t)b)^2\geq 0$.

By an argument in \cite{Ho}, we see that $u\equiv 0$. Therefore, $r\equiv 0$ and so $b^T\tilde{\mathcal R}_\alpha(t)b\equiv 0$. It also follows that $U^+(t)b\equiv 0$ and hence $\dot U^+(t)b\equiv 0$. Therefore, by matrix Riccati equation of $U^+$, we have $\tilde{\mathcal R}_\alpha(t)b\equiv 0$. Finally, we have
\[
\frac{d}{dt}b^TF(t)=-b^T\tilde{\mathcal R}_\alpha(t)E(t)=0.
\]
\end{proof}

\begin{prop}\label{nonpos2}
Let $c$ be a regular value of $H$ and assume that $\Sigma_c=H^{-1}(c)$ is compact. Assume that, for each $\tilde v$ in $\tilde\Lambda_\alpha$ with $\alpha$ in $\Sigma_c$, $|(\widetilde{d\varphi}_t(\tilde v))^h|$ is increasing for each $t>0$ and is decreasing for each $t<0$. If, for each $\alpha$ in $\Sigma_c$, $\tilde{\mathcal R}_\alpha(t)b^T\tilde E(t)<0$ for some $t$, then the Hamiltonian flow is Anosov on $\Sigma_c$.
\end{prop}

\begin{proof}
Suppose that the Hamiltonian flow is not Anosov. By Proposition \ref{anosovJo}, there is a vector $b^TF(0)=b^TF(t)$ in $J(t)$ for all $t$. If we differentiate this equation, then we obtain $\tilde{\mathcal R}_\alpha(t)b^T\tilde E(t)\equiv 0$ which is a contradiction.
\end{proof}

Finally, we remark that Theorem \ref{nonpositive} follows immediately from Proposition \ref{nonpos1} and \ref{nonpos2}.

\smallskip

\section{Entropy estimates}

In this section, we give the proofs of the two entropy estimates, Theorem \ref{entropy1} and \ref{entropy2}.
Let $v$ be in $\mathfrak V$. The positive $\chi^+$ and negative $\chi^-$ Lyapunov exponents are defined by
\[
\chi^\pm(v)=\lim_{t\to\pm\infty}\frac{1}{|t|}\log|\widetilde{d\varphi}_t(v)|.
\]

Let $E^u_\alpha$, $E^s_\alpha$, and $E^0_\alpha$ be the subspaces of $\mathfrak V$ defined by
\[
\begin{split}
&E^u_\alpha=\{v\in\mathfrak V|\chi^-(v)=-\chi^+(v)<0\},\\
&E^s_\alpha=\{v\in\mathfrak V|\chi^+(v)=-\chi^-(v)<0\},\\
&E^0_\alpha=\{v\in\mathfrak V|\chi^-(v)=\chi^+(v)=0\}.
\end{split}
\]

By Oseledets Theorem, $\mathfrak V_\alpha=E^u_\alpha\oplus E^s_\alpha\oplus E^0_\alpha$ holds for $\mu$-almost all $\alpha$.

\begin{proof}[Proof of Theorem \ref{entropy1}]
The same argument as in \cite[Proposition 2.1]{BaWo} shows that the skew orthogonal complement of $E^u_\alpha$ is $E^u_\alpha\oplus E^0_\alpha$. If $v$ is contained in $E^u_\alpha$, then $|\widetilde{d\varphi}_t(v)|$ is bounded for all $t\leq 0$. By Lemma \ref{bddDelta}, $v$ is contained in $\tilde\Delta^-$. Therefore, $E^u_\alpha\subseteq\tilde\Delta^-_\alpha\subseteq E^u_\alpha\oplus E^0_\alpha$.

By Pesin's formula \cite{Pe},
\[
h_\mu=\int_{\Sigma_c}\chi(\alpha)d\mu(\alpha),
\]
where $\chi(\alpha)=\lim_{t\to\infty}\frac{1}{|t|}\log|\det(\widetilde{d\varphi}_t\big|_{\tilde\Delta^-_\alpha})|$. Here determinant is taken with respect to orthonormal frames of any Riemannian metric.

Let $U(s,t)$ be as in the proof of Theorem \ref{Green} and let $U^-(t)=\lim_{s\to-\infty}U(s,t)$. Then $\tilde\Delta^-$ is spanned by the components of
\[
\tilde F_\alpha(0)-U^-_\alpha(0)\tilde E_\alpha(0)=B_\alpha(t)^T\tilde F_\alpha(t)-\dot B_\alpha(t)^T\tilde E_\alpha(t),
\]
where $B_\alpha(\cdot)$ is the solution of $\ddot B_\alpha(t)=-\mathcal{\tilde R}_\alpha(t)B_\alpha(t)$ with $B_\alpha(0)=I$.

If we let $\left<\cdot,\cdot\right>$ be a Riemannian metric on $\Sigma_c$ such that
\[
\tilde F_\alpha(0)-U^-_\alpha(0)\tilde E_\alpha(0)
\]
is orthonormal in $\tilde\Delta^-_\alpha$. If we use this Riemannian metric in the definition of $\chi$, then it follows that
\[
\chi(\alpha)=\lim_{t\to\infty}\frac{1}{t}\log\det B_\alpha(t)=\lim_{t\to\infty}\frac{1}{t}\int_0^t\textbf{tr}(U^-_\alpha(s))ds.
\]

By Birkhoff's ergodic theorem and Pesin's formula, we have
\[
h_\mu=\int_{\Sigma_c}\textbf{tr}(U^-_\alpha(0))d\mu(\alpha).
\]

By Cauchy-Schwarz's inequality, we obtain
\[
h_\mu\leq(n-1)^{1/2}\left(\int_{\Sigma_c}(\textbf{tr}\,U^-_\alpha(0))^2d\mu(\alpha)\right)^{1/2}.
\]

Since $U^-_\alpha(t)=U^-_{\widetilde{d\varphi}_t(\alpha)}(0)$ and $\mu$ is invariant, it follows from the matrix Riccati equation that
\[
h_\mu\leq (n-1)^{1/2}\left(-\int_{\Sigma_c}\tilde{\mathfrak r}_\alpha d\mu(\alpha)\right)^{1/2}.
\]

If equality holds, then $U_\alpha^-(0)$ is constant for $\mu$ almost all $\alpha$. It follows from the Riccati equation that $\tilde{\mathfrak R}$ is constant on the support of $\mu$.
\end{proof}

\begin{proof}[Proof of Theorem \ref{entropy2}]
By \cite[Lemma 3.1]{PaPe}, we have
\[
h_\mu\leq\liminf_{t\to 0}\frac{1}{t}\int_{\Sigma_c}\log(\textbf{ex}(\widetilde{d\varphi}_t))d\mu
\]
where $\textbf{ex}\Phi$ is the expansion of the linear map $\Phi$ defined as
\[
\textbf{ex}\Phi=\sup_S\det\Phi|_S
\]
where the supremum is taken over all nontrivial subspaces $S$.

Let $C(t)$ and $D(t)$ be the matrices defined by
\[
\tilde E_\alpha(0)=-\dot C(t)\tilde E_\alpha(t)+C(t)\tilde F_\alpha(t), \quad \tilde F_\alpha(0)=-\dot D(t)\tilde E_\alpha(t)+D(t)\tilde F_\alpha(t).
\]

The matrices $C(t)$ is a solution to the equation
\begin{equation}\label{eqn}
\ddot C(t)=-\tilde{\mathcal R}_\alpha(t)C(t)
\end{equation}
with initial conditions $C(0)=0$ and $\dot C(0)=-I$.

Similarly, $D(t)$ is a solution of the same equation which satisfies $D(0)=I$ and $\dot D(0)=0$.

It follows that $\widetilde{d\varphi}_t$ sends $\tilde E_\alpha(0)$ and $\tilde F_\alpha(0)$ to
\[
\left(
              \begin{array}{cc}
                -\dot C(t) & -\dot D(t) \\
                C(t) & D(t) \\
              \end{array}
            \right)\left(
                     \begin{array}{c}
                       \tilde E_{\varphi_t(\alpha)}(0) \\
                       \tilde F_{\varphi_t(\alpha)}(0) \\
                     \end{array}
                   \right)
.
\]

Using (\ref{eqn}), we see that
\[
\left(
              \begin{array}{cc}
                -\dot C(t) & -\dot D(t) \\
                C(t) & D(t) \\
              \end{array}
            \right)=\left(
              \begin{array}{cc}
                I & 0 \\
                0 & I \\
              \end{array}
            \right)+t \left(
              \begin{array}{cc}
                0 & \tilde{\mathcal R}(0) \\
                -I & 0 \\
              \end{array}
            \right)+o(t)
\]
as $t\to 0$.

It follows as in \cite{PaPe} that
\[
\textbf{ex}(\widetilde{d\varphi}_t)=1+\frac{t}{2}\sum_{i=1}^{n-1}|\lambda_i-1|+o(t),
\]
where $\lambda_i$ are eigenvalues of the matrix $\tilde{\mathfrak R}$.
\end{proof}

\smallskip

\section{Appendix: On Second Order Equations}\label{On}

In this appendix, we recall some facts on the fundamental solutions of the equation
\begin{equation}\label{2ndscalar}
\ddot a(t)=-\mathcal R(t)a(t).
\end{equation}
The results in this section are well-known. They can be found, for instance, in \cite{Gr,Eb,CoIt}. For the convenience of the readers, we also include the proofs of various results.

Let $B$ be the matrix solution of the equation
\begin{equation}\label{2nd}
\ddot B(t)+\mathcal R(t)B(t)=0
\end{equation}
with initial conditions $B(0)=0$ and $\dot B(0)=I$.

For each time $t$ where $B(t)$ is invertible, we set $S(t):=\dot B(t)B(t)^{-1}$. Then $S(t)$ is a family of symmetric matrices satisfying
\begin{equation}\label{S}
\dot S(t)+S(t)^2+\mathcal R(t)=0.
\end{equation}

Let $D(s,t)$ be defined by
\begin{equation}\label{D}
D(s,t)=B(t)\int_t^sB(\tau)^{-1}(B(\tau)^{-1})^Td\tau.
\end{equation}

From now on, we denote the derivative with respect to $t$ and $s$ by dot and prime, respectively. For instance, $\dot D$ denotes derivative of $D$ with respect to $t$ and $D'$ denotes derivative with respect to $s$.

\begin{lem}\label{D}
The family of matrices $t\mapsto D(s,t)$ is a solution of the equation (\ref{2nd}) which satisfies the boundary conditions
\[
D(s,0)=I, \quad D(s,s)=0, \quad \dot D(s,s)=-(B(s)^{-1})^T.
\]
\end{lem}

\begin{proof}
A computation shows that $t\mapsto D(s,t)$ is a solution of the equation (\ref{2nd}) which satisfies the conditions $D(s,s)=0$ and $\dot D(s,s)=-(B(s)^{-1})^T$. Since the Wronskian $\dot D(s,t)^TB(t)-D(s,t)^T\dot B(t)$ is independent of time $t$, we also have $D(s,0)=I$.
\end{proof}

Let $U(s,t)=\dot D(s,t)D(s,t)^{-1}$. $U(s,t)$ is a solution of the equation
\begin{equation}\label{U}
\dot U(s,t)+U(s,t)^2+\mathcal R(t)=0.
\end{equation}

Next, we apply the following comparison principle of matrix Riccati equations \cite[Theorem 1]{Ro} (see \cite{Ro} for the proof).

\begin{thm}\label{comparison}
Let $A_i(t)$ be a family symmetric matrices. Let $S_i$ be the solution of the matrix Riccati equation
\[
\dot S_i(t)+S_i(t)A_i(t) S_i(t)+\mathcal R_i(t)=0, \quad i=1,2.
\]
Assume that $S_2(t_0)\geq S_1(t_0)$ for some $t_0$ and  $\mathcal R_1(t)\geq \mathcal R_2(t)$, $A_1(t)\geq A_2(t)$ for all $t\geq t_0$. Then
\[
S_2(t)\geq S_1(t)
\]
for all $t\geq t_0$.
\end{thm}

For convenience, we also state the result for $t\leq t_0$.

\begin{thm}\label{comparisonInv}
Let $A_i(t)$ be a family of symmetric matrices. Let $S_i$ be the solution of the matrix Riccati equation
\[
\dot S_i(t)+S_i(t)A_i(t) S_i(t)+\mathcal R_i(t)=0, \quad i=1,2.
\]
Assume that $S_2(t_0)\geq S_1(t_0)$ for some $t_0$ and  $\mathcal R_2(t)\geq \mathcal R_1(t)$, $A_2(t)\geq A_1(t)$ for all $t\leq t_0$. Then
\[
S_2(t)\geq S_1(t)
\]
for all $t\leq t_0$.
\end{thm}

For the rest of this section, we assume that any solution $B(\cdot)$ of the equation (\ref{2nd}) satisfies the following assumption. This assumption is satisfied by certain family of matrices associated to a monotone Hamiltonian system without conjugate point.
\begin{assume}\label{noconjugate}
If $B(t_0)=0$ and $\det \dot B(t_0)\neq 0$ for some $t_0$, then $\det B(t)\neq 0$ for  all $t\neq t_0$.
\end{assume}

Under this assumption, the matrix $U(s,t)$ is invertible whenever $s\neq t$.

\begin{lem}\label{Ubdd}
Assume that $s_1<s_2<0<s_3<s_4$. Then, under Assumption \ref{noconjugate},
\[
U(s_2,t)\geq U(s_1,t)\geq U(s_4,t)\geq U(s_3,t)
\]
for all $t$ in the open interval $(s_2,s_3)$.
\end{lem}

\begin{proof}
Let us give the proof of $U(s_2,t)\geq U(s_1,t)$. Other cases follow by a similar argument. By Lemma \ref{D} and the definition of $U$, the eigenvalues of $U$ blows up as $t$ approaches $s_1$. Therefore, by the matrix Riccati equation (\ref{U}), $\dot U(s_2,t)<0$ for all $t$ near $s_2$. It follows that the eigenvalues of $U(s_2,t)$ goes to $+\infty$ as $t\to s_2^+$ and goes to $-\infty$ as $t\to s_2^-$. In particular, $U(s_2,t)\geq U(s_1,t)$ for all $t>s_2$ and near $s_2$. Therefore, the result follows from Theorem \ref{comparison}.
\end{proof}

It follows from the above lemma that we can define the following
\[
U^{+}(t):=\lim_{s\to+\infty}U(s,t),\quad U^{-}(t):=\lim_{s\to-\infty}U(s,t).
\]

Since both $U^+$ and $U^-$ are solutions of the equation (\ref{S}), the comparison theorem also gives the following estimate.

\begin{lem}\label{Riccatibdd}
Assume that $\mathcal R(t)\geq -k^2I$ for some constant $k>0$. Then
\[
\begin{split}
&k\coth(kt)\,I\geq S(t)> U^{-}(t)\geq U^{+}(t) \\
&\quad (\text{resp. } k\coth(kt)\,I\leq S(t)< U^{+}(t)\leq U^{-}(t))
\end{split}
\]
for all $t>0$ (resp. $t<0$).
\end{lem}

\begin{proof}
By Lemma \ref{Ubdd} and the definitions of $U^+$ and $U^-$, we clearly have $U^{-}(t)\geq U^{+}(t)$ for all $t$. By an argument similar to Lemma \ref{Ubdd}, we see that $S(t)>U^-(t)$ (resp. $S(t)<U^+(t)$) if $t>0$ (resp. $t<0$). The family $t\mapsto k\coth(kt)I$ is a solution of (\ref{S}) with $R(t)=-k^2I$. Therefore, by Theorem \ref{comparison}, $k\coth(kt)\,I\geq S(t)$ for all $t>0$. A similar reasoning shows that $k\coth(kt)\,I\leq S(t)$ for $t<0$.
\end{proof}

Let $D^\pm$ be the solutions of the equation
\[
\dot D^\pm(t)=U^\pm(t)D^\pm(t)
\]
with initial condition $D^\pm(0)=I$.

\begin{lem}\label{nonpositiveRiccati}
Assume that there are non-negative constants $K_1$ and $K_2$ such that $-K_2^2I\geq \mathcal R(t)\geq -K_1^2I$. Then $U^+$ (resp. $U^-$) satisfies the following
\[
-K_2I\geq U^+(t)\geq -K_1I\quad (\text{resp. } K_1I\geq U^-(t)\geq K_2I)
\]
for all $t$ and
\[
|b|e^{-K_1t}\leq |D^+(t)b|\leq |b|e^{-K_2t} \quad (\text{resp. } |b|e^{K_2t}\leq |D^-(t)b|\leq |b|e^{K_1t})
\]
for any vector $b$ and all $t>0$.
\end{lem}

\begin{proof}
We will only prove the case when $K_2>0$ since the case $K_2=0$ is very similar. Note that $K\coth(K_2(t-s))I$ is a solution of (\ref{S}) with $\mathcal R=-K^2I$. Therefore, by Theorem \ref{comparison}, we have
\[
K_2\coth(K_2(t-s))I\geq U(s,t)\geq K_1\coth(K_1(t-s))I
\]
for all $t<s$.

Therefore, if we let $s\to\infty$, then we obtain
\begin{equation}\label{U+bdd}
-K_2I\geq U^+(t)\geq -K_1I.
\end{equation}

It follows from (\ref{U+bdd}) that the Euclidean norm $|D^+(t)b|$ of $D^+(t)b$ satisfies
\[
\begin{split}
\frac{d}{dt}|D^+(t)b|^2&=2\left<U^+(t)D^+(t)b,D^+(t)b\right>\\
&\leq -2K_2|D^+(t)b|^2.
\end{split}
\]
Therefore, it follows that
\[
|D^+(t)b|^2\leq |b|^2e^{-2K_2t}.
\]
\end{proof}

Finally, we show that $|B(t)\,v|$ goes to $+\infty$ uniformly as $t$ goes to $\pm\infty$.

\begin{lem}\label{Jacobiblowup}
Let $B(\cdot)$ be a solution of (\ref{2nd}) with $B(0)=0$, $\dot B(0)=I$, and $R(t)\geq -k^2I$, where $k>0$. Then, for each number $K>0$, there is $T>0$  such that
\[
|B(t)\,v|\geq K\,|v|
\]
for all $t\geq T$ (resp. $t\leq -T$).
\end{lem}

\begin{proof}
Let $D^{+}$ be the solution of (\ref{2nd}) with initial condition $D^{+}(0)=I$ and $\dot D^{+}(0)=U^{+}(0)$. It follows from the definition of $D(s,t)$ that $\lim_{s\to+\infty}D(s,t)=D^{+}(t)$ and
\[
D^{+}(t)=B(t)\int_t^\infty B(\tau)^{-1}(B(\tau)^{-1})^Td\tau.
\]

If we differentiate this equation with respect to time $t$, then we obtain
\[
U^{+}(t)\,D^{+}(t)=-(B(t)^{-1})^T+S(t)D^{+}(t).
\]

Therefore, the following holds
\[
S(t)-U^{+}(t)=(B(t)^{-1})^TM(t)^{-1}B(t)^{-1},
\]
where $M(t)=\int_t^\infty B(\tau)^{-1}(B(\tau)^{-1})^Td\tau$.

It follows from Lemma \ref{Riccatibdd} that there is $t_0>0$ such that
\[
\begin{split}
4k&\geq |\left<S(t)v-U^{+}(t)v,v\right>|\\
&=|\left<M(t)^{-1}B(t)^{-1}v,B(t)^{-1}v\right>|\\
&\geq \frac{|B(t)^{-1}v|^2}{||M(t)||}
\end{split}
\]
for all $t>t_0$. Here $||M(t)||$ denotes the operator norm of $M(t)$.

Therefore, for all $v$ satisfying $|v|=1$, we have
\[
|B(t)v|\geq\frac{1}{||B(t)^{-1}||}\geq \frac{1}{(4k||M(t)||)^{1/2}}\to\infty
\]
as $t\to\infty$.
\end{proof}

\end{document}